\documentclass[12pt,a4paper]{article}

\usepackage[latin1]{inputenc}
\usepackage{amsmath}
\usepackage{amsfonts}
\usepackage{amssymb}

\DeclareMathAlphabet{\mathpzc}{OT1}{pzc}{m}{it}

\author{Pedro Caro \footnote{Department of Mathematics, Universidad Aut\'onoma de Madrid, 28049 Madrid, Spain}}
\title{On an inverse problem in electromagnetism with local data: stability and uniqueness}
\date{May 26, 2010}

\newcommand{\bouU}{\partial U} 
\newcommand{\inte}{\Omega}
\newcommand{\curl}{\nabla \! \times \!}
\newcommand{\diver}{\nabla \cdot}
\newcommand{\Durl}{D \times \!}
\newcommand{\Diver}{D \cdot}
\newcommand{\bou}{\partial\Omega} 
\newcommand{\Div}{\textrm{Div}} 

\newcommand{\frv}[4]{\left(\begin{array}{c} 
#1\\ #2\\
\hline
#3\\ #4\\
\end{array}\right)}
\newcommand{\fcv}[4]{\left(\begin{array}{c c|c c} 
#1 & #2 & #3 & #4
\end{array}\right)}
\newcommand{\fbM}[4]{\left(\begin{array}{c|c} 
#1 & #2\\
\hline
#3 & #4
\end{array}\right)}
\newcommand{\subM}[4]{\begin{array}{c c} 
#1 & #2\\
#3 & #4
\end{array}}
\newcommand{\inner}[2]{#1 \! \cdot \! #2} 
\newcommand{\innerf}[2]{\langle #1, #2 \rangle} 
\newcommand{\Inner}[2]{\left( #1 \middle| #2 \right)} 
\newcommand{\norm}[3]{\left\|#1\right\|^{#2}_{#3}}  
\newcommand{\cross}[2]{#1 \! \times \! #2} 
\newcommand{\crossf}[2]{\ast ( #1 \wedge #2 )} 
\newcommand{\dual}[2]{\left\langle #1 \middle| #2 \right\rangle} 
\newcommand{\Hdiv}[1]{H^{#1}(\inte; \textrm{div})} 
\newcommand{\Hcurl}[1]{H^{#1}(\inte; \textrm{curl})} 
\newcommand{\HcurlU}[1]{H^{#1}(U; \textrm{curl})} 
\newcommand{\ssec}[1]{\Lambda^{#1} T \mathbb{E}} 

\newtheorem{theorem}{Theorem}
\newtheorem{definition}{Definition}
\newtheorem{lemma}[definition]{Lemma}
\newtheorem{proposition}[definition]{Proposition}
\newtheorem{corollary}{Corollary}
\newenvironment{proof}{\textbf{Proof:}} {\hspace{\stretch{1}}$\Box$}

\begin{document}

\maketitle

\begin{abstract}
In this paper we prove a stable determination of the coefficients of the time-harmonic Maxwell equations from local boundary data. The argument --due to Isakov-- requires some restrictions on the domain.
\end{abstract}

\section*{Introduction}
Let $ \inte $ be a bounded Lipschitz domain in the three-dimensional euclidean space. Assume the medium, modeled by $ \inte $, to be non-homogeneous and isotropic. Suppose the electromagnetic properties of $ \inte $ to be described by the electric permittivity $ \varepsilon $, the magnetic permeability $ \mu $ and the electric conductivity $ \sigma $. Let $ E, H $ denote the electric and magnetic fields, respectively. The time-harmonic Maxwell equations at frequency $ \omega > 0 $ read
\[ dH + i \omega \varepsilon \ast\! E = \sigma \ast\! E, \qquad dE - i \omega \mu \ast\! H = 0, \]
whenever the total electric current density is given by $ \sigma \ast\! E $.
Writing $ \gamma = \varepsilon + i \sigma/ \omega $, the time-harmonic Maxwell equations can be expressed as
\begin{equation}\label{ME_without_sources}
\left\{ 
\begin{array}{l}
dH + i \omega \gamma \ast\! E = 0\\
dE - i \omega \mu \ast\! H = 0.\\
\end{array}
\right.
\end{equation}
It is known that this system may present positive \emph{resonant frequencies} even when the domain is assumed to be of class $ C^2 $ (see \cite{SIsCh}) or $ \sigma $ is assumed to vanish (see \cite{Le}). This means that, for some positive frequencies, the system may have non-trivial solutions of (\ref{ME_without_sources}) with zero boundary conditions.

In these notes, we shall study the inverse boundary value problem of determining in a stable manner the coefficients $ \mu, \varepsilon, \sigma $ by local boundary measurements. When setting this problem, the possible existence of resonant frequencies makes natural the use of the restricted Cauchy data set as a model of non-invasive measurements, instead of using either the restricted admittance or impedance maps. Cauchy data sets have been used successfully in \cite{BuU}, \cite{SaTz}, \cite{SaTz2} and \cite{C}.

Let $ \Gamma $ be a proper non-empty open subset of $ \bou $, the boundary of $ \inte $. Let $ \nu $ be $ 1 $-form defined by $ \nu = e(N,\centerdot) $ with $ N $ the outward unit vector normal to $ \bou $ and $ e $ the euclidean metric. Given a frequency $ \omega > 0 $, the \emph{Cauchy data set restricted to $ \Gamma $} is defined as follows: $ (T, S) \in C(\mu, \gamma; \Gamma) $ if and only if $ T \in TH_0(\Gamma) $, $ S \in TH(\Gamma) $, and there exists a pair $ (E, H) \in (\Hcurl{})^2 $ solution of (\ref{ME_without_sources}) satisfying $ \crossf{\nu}{E} = T $ and $ \crossf{\nu}{H}|_{\Gamma} = S $.

In order to quantify the proximity of the restricted Cauchy data sets we introduce a pseudo-metric distance which was already used in \cite{C}.

\begin{definition}\label{def:Cdistance_res}\rm
Let  $ \mu_1, \gamma_1 $ and $ \mu_2, \gamma_2 $ be two pairs of coefficients. Consider $ \omega $ a positive frequency and let $ C^j_\Gamma $ denote $ C(\mu_j, \gamma_j; \Gamma) $. Let us define the pseudo-metric distance between the restricted Cauchy data sets $ C^1_\Gamma $ and $ C^2_\Gamma $ as
\begin{equation*}
\delta_C (C^1_\Gamma, C^2_\Gamma) = \max_{j \neq k} \sup_{ \substack{(T_k, S_k) \in C^k_\Gamma \\ \norm{T_k}{}{TH_0(\Gamma)} = 1 }} \inf_{(T_j, S_j) \in C^j_\Gamma} \norm{(T_j, S_j) - (T_k, S_k)}{}{TH_0(\Gamma) \times TH(\Gamma)}.
\end{equation*}
\end{definition}

In order to state our result, we need stable determination of the problem on the boundary. Since this has not been proven yet, we shall introduce the following definitions.
\begin{definition}\label{def:admissible_loc}\rm
Given two constants $ M, s $ such that $ 0 < M $, $ 0 < s < 1/2 $, we shall say that the pair of coefficients $ \mu, \gamma $ is \emph{admissible} if they satisfy the following conditions.
\begin{itemize}
\item[(i)] \emph{Uniform ellipticity condition.} The coefficients $ \gamma, \mu \in C^{1,1}(\overline{\inte}) $ satisfy
\[ M^{-1} \leq \mathrm{Re}\,\gamma(x) \qquad M^{-1} \leq \mu(x); \]
for any $ x \in \inte $.
\item[(ii)] \emph{A priori bound on the boundary.} The following a priori bound holds on the boundary
\begin{equation*}
\norm{\gamma}{}{C^{0,1}(\overline{\Gamma})} + \norm{\mu}{}{C^{0,1}(\overline{\Gamma})} < M.
\end{equation*}
\item[(iii)] \emph{A priori bound in the interior.} The following a priori bounds hold in the interior
\begin{align*}
\norm{\gamma}{}{W^{2,\infty}(\inte)} + \norm{\mu}{}{W^{2,\infty}(\inte)} \leq M,& &
\norm{\gamma}{}{H^{2 + s}(\inte)} + \norm{\mu}{}{H^{2 + s}(\inte)} \leq M.
\end{align*}
\end{itemize}
\end{definition}

\begin{definition}\rm
Let $ M, s $ be the constants given in Definition \ref{def:admissible_loc} and let $ \omega $ be a positive frequency. We shall say that a pair $ \mu, \gamma $ is in the \emph{class of $ B $-stable coefficients on $ \Gamma $ at frequency $ \omega $} if $ \mu, \gamma $ is an admissible pair and there exists a modulus of continuity $ B $ such that, for any other admissible pair $ \tilde{\mu}, \tilde{\gamma} $, one has
\begin{gather*}
\norm{\gamma - \tilde{\gamma}}{}{C^{0,1}(\overline{\Gamma})} + \norm{\mu - \tilde{\mu}}{}{C^{0,1}(\overline{\Gamma})} \leq B \left( \delta_C(C_\Gamma, \tilde{C}_\Gamma) \right),\\
\norm{\nabla (\gamma - \tilde{\gamma})}{}{L^\infty(\Gamma; \mathbb{C}^3)} + \norm{\nabla (\mu - \tilde{\mu})}{}{L^\infty(\Gamma; \mathbb{C}^3)} \leq B \left( \delta_C(C_\Gamma, \tilde{C}_\Gamma) \right).
\end{gather*}
\end{definition}
Here $ C_\Gamma, \tilde{C}_\Gamma $ are the Cauchy data sets associated to the pairs $ \mu, \gamma $ and $ \tilde{\mu}, \tilde{\gamma} $, respectively.

The first idea in our argument is to construct special solutions vanishing on $ \bou \setminus \Gamma $, the inaccessible part of the boundary. In \cite{I2} Isakov proposed a reflection argument which allows to construct solutions for the conductivity equation with the desired behavior on the boundary. This argument was extended in \cite{COSa} to the time-harmonic Maxwell equation.

In order to carry out Isakov's approach it seems to be necessary to assume some geometrical restrictions about the domain, namely, the inaccessible part is supposed to be either part of a plane or part of a sphere. Despite this restriction, the method allows to prove the following result from local boundary data.

\begin{theorem}\label{th:stability_loc}\rm
Let $ U $ be either a \emph{suitable partially flat domain} or a \emph{suitable partially spherical domain} and let $ \omega $ be a positive frequency. Consider $ \mu_1, \gamma_1 $ and $ \mu_2, \gamma_2 $ any two pairs in the class of $ B $-stable coefficients on $ \Gamma $ at frequency $ \omega $, with $ B $ satisfying $ |r| \leq B(|r|) $ for all $ |r| < 1 $. Assume that $ \partial_N \mu_j = \partial_N \gamma_j = 0 $ on $ \bouU \setminus \overline{\Gamma} $ with $ j = 1,2 $. Then, there exists a constant $ C = C(M) $ such that the following estimate holds
\[ \norm{\gamma_1 - \gamma_2}{}{H^1(U)} + \norm{\mu_1 - \mu_2}{}{H^1(U)} \leq C\, |\mathrm{log}\, B(\delta_C(C^1_\Gamma, C^2_\Gamma))|^{-\lambda}, \]
for some constant $ \lambda $ such that $ 0 < \lambda < s^2/3 $. Here $ C^1_\Gamma, C^2_\Gamma $ are the restricted Cauchy data sets associated to the pairs $ \mu_1, \gamma_1 $ and $ \mu_2, \gamma_2 $, respectively.
\end{theorem}
The exact meaning of \emph{suitable} is explained in Subsection \ref{sec:geometryU}.

As in the inverse conductivity problem, it should be possible to prove that any admissible pair is in the class of H\"older-stable coefficients on $ \Gamma $ for any frequency $ \omega $, that is, with $ B(|r|) = |r|^\alpha $ for $ 0 < \alpha < 1 $. Note that we have obtained the same kind of stability as in the global data case (see \cite{C}).

From the point of view of applications it might be useful to suppose the coefficients to be equal on the accessible part of the boundary. In this particular case we get the following corollary.

\begin{corollary}\label{col:stability_loc}\rm
Let $ U $ be either a \emph{suitable partially flat domain} or a \emph{suitable partially spherical domain} and let $ \omega $ be a positive frequency. Consider $ \mu_1, \gamma_1 $ and $ \mu_2, \gamma_2 $ any two pairs of admissible coefficients. Assume that
\[ \mu_1|_{\Gamma} = \mu_2|_{\Gamma}, \, \partial_{x^j} \mu_1|_{\Gamma} = \partial_{x^j} \mu_2|_{\Gamma},\quad \gamma_1|_{\Gamma} = \gamma_2|_{\Gamma},\, \partial_{x^j} \gamma_1|_{\Gamma} = \partial_{x^j} \gamma_2|_{\Gamma} \]
and $ \partial_N \mu_k = \partial_N \gamma_k = 0 $ on $ \bouU \setminus \overline{\Gamma} $ with $ j = 1,2,3 $ and $ k = 1,2 $. Then, there exists a constant $ C = C(M) $ such that the following estimate holds
\[ \norm{\gamma_1 - \gamma_2}{}{H^1(U)} + \norm{\mu_1 - \mu_2}{}{H^1(U)} \leq C |\mathrm{log}\, \delta_C(C^1_\Gamma, C^2_\Gamma)|^{- \lambda}, \]
for some constant $ \lambda $ such that $ 0 < \lambda < s^2/3 $.
\end{corollary}

Furthermore, if we follow the proof of Theorem \ref{th:stability_loc} one can state the following uniqueness result.

\begin{theorem}\label{th:uniqueness_loc}\rm
Let $ U $ be either a \emph{suitable partially flat domain} or a \emph{suitable partially spherical domain} and let $ \omega $ be a positive frequency. Consider $ \mu_1, \gamma_1 $ and $ \mu_2, \gamma_2 $ in $ C^{1,1}(\overline{U}) $ such that
\[ \mu_1|_{\Gamma} = \mu_2|_{\Gamma}, \, \partial_{x^j} \mu_1|_{\Gamma} = \partial_{x^j} \mu_2|_{\Gamma},\quad \gamma_1|_{\Gamma} = \gamma_2|_{\Gamma},\, \partial_{x^j} \gamma_1|_{\Gamma} = \partial_{x^j} \gamma_2|_{\Gamma}, \]
with $ j = 1,2,3 $. If additionally $ C^1_\Gamma = C^2_\Gamma $ and $ \partial_N \mu_k = \partial_N \gamma_k = 0 $ on $ \bouU \setminus \overline{\Gamma} $ with $ k = 1,2 $, then
\[ \mu_1 = \mu_2,\qquad \gamma_1 = \gamma_2 \]
in $ U $.
\end{theorem}

As in the inverse conductivity problem, it should be possible to prove that the coefficients are equal on the accessible part of the boundary $ \Gamma $ whenever $ C^1_\Gamma = C^2_\Gamma $.

The problem of determining the electromagnetic coefficients by data taken on the entire boundary has been studied by several authors. The unique recovery of $ C^3 $-coefficients $ \gamma $ and $ \mu $ from boundary data was proved in \cite{OPS}, and later simplified in \cite{OS}. Boundary determination results were given in \cite{JMcD} in the case that the boundary is smooth. The more general chiral media was studied in \cite{McD2}.  For a slightly more general approach and more background information, see also the review article \cite{OPS2}.

The inverse problem of determining the electromagnetic coefficients from partial data has been much less considered. As far as the author knows the only work in that direction is \cite{COSa}.

Two different approaches have been used to attack the inverse conductivity problem from partial boundary data. The first one was proposed in \cite{BuU} and generalized in \cite{KSjU}. In \cite{HeW}, this method was used to give a log-log-stable determination in the framework of \cite{BuU}. In this approach there are not any strong geometrical restriction about the domain but the partial measurements have to be taken in the whole boundary. Getting an optimal stability (i. e. a stability with a log-type modulus of continuity) in the context of \cite{BuU} may be difficult and the stability for \cite{KSjU} is an open question. The second approach for partial data was proposed in \cite{I2} and the optimal stable determination was stated in \cite{HeW2}. As we have already mentioned, this argument requires a strong restriction on the domain. However, the measurements are localized on the accessible part of the boundary and it is possible to get the optimal stable determination. These two facts are very important from the point of view of applications. For instance, Alessandrini and Vessella proved in \cite{AV} that a logarithmic estimate yields Lipschitz stability for some finite dimensional spaces of conductivities.

These two approaches have been extended to systems. In \cite{SaTz2} Salo and Tzou followed the spirit of \cite{KSjU} to prove uniqueness in the context of Dirac's equation. Isakov's argument was extended in \cite{COSa} to Maxwell's equations. The proof given here takes some ideas from \cite{KSaU} and it turns out to be more convenient and useful for us than the proof given in \cite{COSa}. In fact, it avoids the long computations made there to prove the thesis of Theorem \ref{th:uniqueness_loc} and it allows to relax the hypothesis about the domain and the smoothness of the coefficients. In \cite{COSa} the domain was assumed to be of class $ C^{1,1} $ and the coefficients were assumed to be $ C^4 $. Besides, a technical hypothesis about the extension of the coefficients had to be supposed.

An overview of the paper is the following. It has three sections. In the first one, we give some preliminaries about the functional spaces used. In the second section we prove our results when $ U $ is partially flat. To achieve this, we use a reflection argument to construct special solutions vanishing on the inaccessible part of the boundary. In the third section we connect the flat case with the spherical one by means of the Kelvin transform.

\paragraph{Acknowledgement.} This paper is part of the author's doctoral dissertation and it has been written under the supervision of Alberto Ruiz. The author would like to thank him for his support and dedication. The author would also like to thank Petri Ola and Mikko Salo for their invitation to inverse problems in electromagnetism. The author was economically supported by Ministerio de Ciencia e Innovaci\'on de Espa\~na, MTM2008-02568-C02-01.

\section{Preliminaries}
Let $ \mathbb{E} $ be the three-dimensional euclidean point space and let its tangent bundle be denoted by   $ T\mathbb{E} $. Let $ \mathcal{T}\mathbb{E} $ be the module of smooth vector fields over the real smooth functions $ C^\infty(\mathbb{E}; \mathbb{R}) $ and define
\[ \mathcal{X} \mathbb{E} = \{ u + iv : u,v \in \mathcal{T} \mathbb{E} \} .\]
The elements of $ \mathcal{X} \mathbb{E} $ will be called complex vector fields. Let the bundle of alternating tensors be denoted by $ \ssec{k} $ with $ k = 0,1,2,3 $. Let $ \mathcal{A}^k \mathbb{E} $ be the vector space of differential $ k $-forms and define
\[ \Lambda^k \mathbb{E} = \{ \omega + i\eta : \omega,\eta \in \mathcal{A}^k \mathbb{E} \}. \]
The elements of $ \Lambda^k \mathbb{E} $ will be called complex $ k $-forms. Recall that $ 0 $-forms are smooth functions by definition. As it is usual, $ d $ and $ \wedge $ denote the exterior derivative operator and the exterior product of forms, respectively.

The euclidean metric $ e $ induces a volume element denoted by $ dV $, a distance denoted by $ d_e $ and a point-wise inner product denoted by $ \langle \omega, \eta \rangle $ for any $ \omega, \eta \in \mathcal{A}^k \mathbb{E} $ with $ k=0,1,2,3 $. Recall that the Hodge star operator is the unique bundle map $ \ast : \ssec{k} \longrightarrow \ssec{3-k} $ satisfying
\[ \omega \wedge \ast \eta = \langle \omega, \eta \rangle \, dV.\]
Moreover, $ \ast \ast \omega = \omega $. Let us define $ |\eta|^2 = \innerf{\eta}{\eta} $.

The formal adjoint of $ d $ will be denoted by $ \delta $ and it can be expressed by
\[ \delta \eta = (-1)^k \ast d \ast \eta \]
for $ \eta \in \mathcal{A}^k $. Let us define the laplacian on $ k $-forms as $ -\Delta := \delta d + d \delta $.

We also recall that we can identify vectors and $ 1 $-forms by means of the metric, that is,
\begin{equation}\label{eq:vectors-covectors}
u \in \mathcal{T} \mathbb{E} \longmapsto \eta = e(u,\centerdot) \in \mathcal{A}^1 \mathbb{E}.
\end{equation}
If $ u \in \mathcal{T} \mathbb{E} $, its corresponding $ 1 $-form will be denoted by $ u^\flat $. However if the difference is clear by the context it will be denoted by $ u $. On the other hand, if $ v \in \mathcal{A}^1 \mathbb{E} $, its corresponding vector field will be denoted by $ v^\sharp $. As before, this notation will be used whenever the context is not clear.

Finally, for any $ f \in C^\infty(\mathbb{E}; \mathbb{R}) $ and any $ u, v \in \mathcal{T} \mathbb{E} $, $ u \cdot v = e(u, v) $ denotes the point-wise inner product, $ u \times v = \ast(u^\flat \wedge v^\flat)^\sharp $ denotes the point-wise cross product and $ \nabla f = (df)^\sharp $, $ \diver u = - \delta u^\flat $ and $ \curl u = (\ast du^\flat)^\sharp $ stand for the gradient, divergence and curl, respectively.

\subsection{The functional spaces}
Along these notes we shall say that a domain $ \inte $ is Lipschitz if its boundary $ \bou $ is locally the graph of a Lipschitz function. Additionally, $ N $ denotes the outward unit vector normal to $ \bou $ and $ \nu := N^\flat $ is its corresponding $ 1 $-form.

In order to perform the proofs of the results stated in the introduction we require some standard functional spaces: $ H^s(\mathbb{E}), H^s(\inte), H^s_0(\inte) $ with $ s \in \mathbb{R} $ denote the potential Sobolev spaces based in $ L^2 $; $ W^{2, \infty}(\inte) $ stands for the Sobolev space with two derivaties in $ L^\infty $; $ B^{s}(\bou) $ with $ 0 < |s| < 1 $ denotes the Besov spaces $ B^s_{p,q}(\bou) $ with $ p = q = 2 $. A quite complete description of these spaces can be found in \cite{JK}.

Additionally, when working with Maxwell's equations other non-standard Sobolev and Besov spaces turn to be useful. Those are mainly, $ \Hdiv{} $, $ \Hcurl{} $ and $ TH(\bou) $. The first one corresponds to the fields or $ 1 $-forms in $ L^2 $ with divergence in $ L^2 $. In this space, it makes sense the normal traces as elements of $ B^{-1/2}(\bou) $. The space $ \Hcurl{} $ corresponds to the fields or $ 1 $-forms in $ L^2 $ whose curl is in $ L^2 $. The tangential traces of elements of $ \Hcurl{} $ make sense as elements of $ B^{-1/2}(\bou; \mathbb{C}^3) $ or $ B^{-1/2}(\bou; \ssec{1}) $. The space $ TH(\bou) $ is defined as the space of tangential traces of $ \Hcurl{} $ and we have that
\begin{gather}
\label{es:equivalent_norm_TH-C1}
\norm{w}{}{B^{-1/2}(\bou; \mathbb{C}^3)} + \norm{\Div \, w}{}{B^{-1/2}(\bou)} \leq C_1 \norm{w}{}{TH(\bou)}, \\
\label{es:equivalent_norm_TH-C2}
\norm{w}{}{TH(\bou)} \leq C_2 \left( \norm{w}{}{B^{-1/2}(\bou; \mathbb{C}^3)} + \norm{\Div \, w}{}{B^{-1/2}(\bou)} \right),
\end{gather}
if $ w $ is a vector field, or
\begin{gather}
\label{es:equivalent_norm_TH-C1_forms}
\norm{w}{}{B^{-1/2}(\bou; \Lambda^1 T\mathbb{E})} + \norm{\Div \, w}{}{B^{-1/2}(\bou)} \leq C_1 \norm{w}{}{TH(\bou)}, \\
\label{es:equivalent_norm_TH-C2_forms}
\norm{w}{}{TH(\bou)} \leq C_2 \left( \norm{w}{}{B^{-1/2}(\bou; \Lambda^1 T\mathbb{E})} + \norm{\Div \, w}{}{B^{-1/2}(\bou)} \right),
\end{gather}
if $ w $ is a $ 1 $-form. Here $ \Div $ stands for the surface divergence. Recall that the surface divergence of $ w \in TH(\bou) $ makes sense, and it can be defined as an element of $ B^{1/2}(\bou) $ in the following way:
\[ \Div \, w = -\inner{N}{(\curl u)},\qquad \Div \, w = -\langle \nu, \ast du \rangle \]
where $ u \in \Hcurl{} $ and $ \cross{N}{u} = w $, if $ u $ is a vector field, or $ \ast( \nu \wedge u ) = w $, if $ u $ is a $ 1 $-form.

Finally, we recall some key points. For any $ u \in \Hdiv{} $ and any $ g \in B^{1/2}(\bou) $ we have
\begin{equation}\label{def:normal-component}
\dual{\inner{N}{u}}{g} = \int_\inte (\diver u) \overline{f} \, dV + \int_\inte \inner{u}{\overline{\nabla f}} \, dV,
\end{equation}
where $ u $ is a vector field, $ f \in H^1(\inte) $ and $ f|_{\bou} = g $. We also have
\begin{equation}\label{def:normal-component_forms}
\dual{\langle \nu, u \rangle}{g} = - \int_\inte (\delta u) \overline{f} \, dV + \int_\inte \langle u, \overline{d f} \rangle \, dV,
\end{equation}
where $ u $ is a $ 1 $-form, $ f \in H^1(\inte) $ and $ f|_{\bou} = g $.

The maps
\[ \cross{N}{\centerdot} : TH(\bou) \longrightarrow (TH(\bou))^\ast, \qquad \ast(\nu \wedge \centerdot) : TH(\bou) \longrightarrow (TH(\bou))^\ast \]
are isomorphisms. In particular,
\begin{equation}\label{for:duality_TH-TH*}
\int_\inte \inner{(\curl u)}{\overline{v}} \, dV = \int_\inte \inner{u}{(\overline{\curl v})} \, dV -\dual{\cross{N}{u}}{\cross{N}{(\cross{N}{v})}},
\end{equation}
for vector fields, and
\begin{equation}\label{for:duality_TH-TH*_forms}
\int_\inte \langle \ast d u, \overline{v} \rangle \, dV = \int_\inte \langle u, \overline{\ast d v} \rangle \, dV -\dual{\ast(\nu \wedge u)}{\ast(\nu \wedge \ast(\nu \wedge v))},
\end{equation}
for $ 1 $-forms.

A detailed exposition of the collected facts can be seen in \cite{M} and \cite{C}.

\subsection{Some remarks on the boundary}
Along this section, $ \inte $ denotes any bounded Lipschitz domain, $ \bou $ denotes its boundary and $ \Gamma $ stands for a proper non-empty open subset of $ \bou $. Moreover, $ \centerdot |_\Gamma $ and $ \centerdot |_{\overline{\Gamma}} $ denote the restriction to $ \Gamma $ and $ \overline{\Gamma} $, respectively.

\begin{definition}\rm For $ 0 < s < 1 $, define the space $ B^{s}(\Gamma) $ as
\[ B^s(\Gamma) = \{ f|_\Gamma : f \in B^s(\bou) \}, \]
with the norm
\[ \norm{g}{}{B^s(\Gamma)} = \textrm{inf} \{ \norm{f}{}{B^s(\bou)} : f|_\Gamma = g \}. \]
On the other hand, for $ 0 < |s| < 1 $, define
\[ B^s_0(\Gamma) = \{ f \in B^s(\bou) : \mathrm{supp}\, f \subset \overline{\Gamma} \}, \]
with norm
\[ \norm{f}{}{B^s_0(\Gamma)} = \norm{f}{}{B^s(\bou)}. \]
Finally, for $ 0 < s < 1 $, define the space $ B^{-s}(\Gamma) $ as the dual of $ B^s_0(\Gamma) $, that is,
\[ B^{-s}(\Gamma) = (B^s_0(\Gamma))^*. \]
\end{definition}
Note that, for $ 0 < s < 1 $, $ B^{-s}_0(\Gamma) $ is the dual space of $ B^s(\Gamma) $, that is,
\[ B^{-s}_0(\Gamma) = (B^s(\Gamma))^*. \]

\begin{lemma}\label{le:besov_product}\rm
Let $ s, \epsilon $ be such that $ 0 < s < 1 $ and $ 0 < \epsilon \leq 1-s $. Then there exists a constant $ C(s,\epsilon) > 0 $ such that,
\begin{itemize}
\item[(a)] for any $ g \in C^{0,s+\epsilon}(\bou) $ and any $ f \in B^s(\bou) $,
\begin{equation}\label{es:besov_product}
\norm{gf}{}{B^s(\bou)} \leq C \norm{g}{}{C^{0,s+\epsilon}(\bou)} \norm{f}{}{B^s(\bou)};
\end{equation}
\item[(b)] for any $ g : \bou \longrightarrow \mathbb{C} $ with $ g \in C^{0,s+\epsilon}(\overline{\Gamma}) $ and any $ f \in B^s_0(\Gamma) $,
\begin{equation}\label{es:besov_product_0}
\norm{gf}{}{B^s_0(\Gamma)} \leq C \norm{g}{}{C^{0,s+\epsilon}(\overline{\Gamma})} \norm{f}{}{B^s_0(\Gamma)};
\end{equation}
\item[(c)] for any $ g : \bou \longrightarrow \mathbb{C} $ with $ g \in C^{0,s+\epsilon}(\overline{\Gamma}) $ and any $ f \in B^s(\Gamma) $,
\begin{equation}\label{es:besov_product_Gamma}
\norm{g|_\Gamma f}{}{B^s(\Gamma)} \leq C \norm{g}{}{C^{0,s+\epsilon}(\overline{\Gamma})} \norm{f}{}{B^s(\Gamma)}.
\end{equation}
\end{itemize}
\end{lemma}

\textbf{Remark:} The constant $ C(s,\epsilon) $ given here blows up when $ \epsilon $ becomes small.

\begin{proof}
(a) was proven in \cite{C}. (b) and (c) follow easily from (a) taking $ \tilde{g} \in C^{0,s+\epsilon}(\bou) $ an extension of $ g|_{\overline{\Gamma}} $ such that
\[ \norm{\tilde{g}}{}{C^{0,s + \epsilon}(\bou)} \leq C \norm{g}{}{C^{0,s + \epsilon}(\overline{\Gamma})}. \]
Regarding to extensions from a closed subset of $ \mathbb{E} $, see \cite{St}.
\end{proof}

\begin{definition}\rm Define the space $ TH(\Gamma) $ as
\[ TH(\Gamma) = \{ w|_\Gamma : w \in TH(\bou) \}, \]
with the norm
\[ \norm{z}{}{TH(\Gamma)} = \textrm{inf} \{ \norm{w}{}{TH(\bou)} : w|_\Gamma = z \}. \]
On the other hand, define
\[ TH_0(\Gamma) = \{ w \in TH(\bou) : \mathrm{supp}\, w \subset \overline{\Gamma} \}, \]
with norm
\[ \norm{w}{}{TH_0(\Gamma)} = \norm{w}{}{TH(\bou)}. \]
\end{definition}

\begin{lemma}\rm
Let $ N $ be the outward unit vector normal to $ \bou $ and let $ \nu $ be its associated $ 1 $-form. Then
\[ \cross{N}{TH_0(\Gamma)} = (TH(\Gamma))^*, \qquad \ast(\nu \wedge TH_0(\Gamma)) = (TH(\Gamma))^*. \]
\end{lemma}

\begin{proof}
Here we prove the first identity. The second one follows by the correspondence between vector fields and $ 1 $-forms.

Let $ l : TH(\Gamma) \rightarrow \mathbb{C} $ be a bounded linear functional, we can construct another functional $ \tilde{l} : TH(\bou) \rightarrow \mathbb{C} $ defined by $ \tilde{l}(w) = l(w|_\Gamma) $, for any $ w \in TH(\bou) $. Since $ \tilde{l} $ is linear, bounded and
\[ \norm{\tilde{l}}{}{(TH(\bou))^*} \leq \norm{l}{}{(TH(\Gamma))^*}, \]
there exists $ z \in TH(\bou) $ such that $ \dual{\cross{N}{z}}{w} = \tilde{l}(w) = l(w|_\Gamma) $ with
\[ \norm{z}{}{TH(\bou)} \leq \norm{l}{}{(TH(\Gamma))^*}. \]
It is clear that $ \mathrm{supp}\, \cross{N}{z} \subset \overline{\Gamma} $, hence $ z \in TH_0(\Gamma) $ and
\[ \norm{z}{}{TH_0(\Gamma)} \leq \norm{l}{}{(TH(\Gamma))^*}. \]

Conversely, given $ z \in TH_0(\Gamma) $ we can define $ l : TH(\Gamma) \rightarrow \mathbb{C} $ as $ l(w) = \dual{\cross{N}{z}}{\tilde{w}} $, for any $ w \in TH(\Gamma) $ and $ \tilde{w} \in TH(\bou) $ such that $ \tilde{w}|_\Gamma = w $. It is well-defined since $ \mathrm{supp}\, \cross{N}{z} \subset \overline{\Gamma} $.
Moreover,
\[ |l(w)| \leq \norm{z}{}{TH_0(\Gamma)} \norm{\tilde{w}}{}{TH(\bou)}, \]
which implies
\[ |l(w)| \leq \norm{z}{}{TH_0(\Gamma)} \norm{w}{}{TH(\Gamma)}. \]
Therefore, $ l $ is a bounded linear operator with norm
\[ \norm{l}{}{(TH(\Gamma))^*} \leq \norm{z}{}{TH_0(\Gamma)}. \]
\end{proof}

\begin{lemma}\rm
There exists a positive constant $ C $ such that:
\begin{itemize}
\item[(a)] For any $ w \in TH(\bou) $ and any $ f \in C^{0,1}(\bou) $, one has that
\begin{equation}\label{es:TH-product}
\norm{fw}{}{TH(\bou)} \leq C \norm{f}{}{C^{0,1}(\bou)} \norm{w}{}{TH(\bou)}.
\end{equation}
\item[(b)] For any $ w \in TH_0(\bou) $ and any $ f : \bou \longrightarrow \mathbb{C} $ such that $ f \in C^{0,1}(\overline{\Gamma}) $, one has that
\begin{equation}\label{es:TH-product_0}
\norm{fw}{}{TH_0(\Gamma)} \leq C \norm{f}{}{C^{0,1}(\overline{\Gamma})} \norm{w}{}{TH_0(\Gamma)}.
\end{equation}
\end{itemize}
\end{lemma}

\begin{proof}
(a) was proven in \cite{C}. (b) follows easily from (a) taking and extension $ \tilde{f} $ of $ f|_{\overline{\Gamma}} $ such that $ \tilde{f} \in C^{0,1}(\bou) $ and satisfying
\[ \norm{\tilde{f}}{}{C^{0,1}(\bou)} \leq C \norm{f}{}{C^{0,1}(\overline{\Gamma})}. \]
Again, regarding to extensions from a closed subset of $ \mathbb{E} $, see \cite{St}.
\end{proof}

\begin{lemma}\rm
The following items hold:
\begin{itemize}
\item[(a)] If $ w \in TH_0(\Gamma) $, then $ \Div\, w \in B^{-1/2}_0(\Gamma) $ and
\begin{equation}\label{es:Div-boun-TH0}
\norm{\Div\, w}{}{B^{-1/2}_0(\Gamma)} 
\leq C \norm{w}{}{TH_0(\Gamma)};
\end{equation}
\item[(b)]If $ z \in TH(\bou) $, then, for any $ f \in B^{1/2}_0(\Gamma) $ and $ \tilde{z} \in TH(\bou) $ such that $ \tilde{z}|_\Gamma = z|_\Gamma $, one has $ \dual{(\Div\, z)|_\Gamma}{f} = \dual{\Div\, \tilde{z}}{f} $ and
\begin{equation}\label{es:Div-boun-THGamma}
\norm{(\Div\, z)|_\Gamma}{}{B^{-1/2}(\Gamma)} 
\leq C \norm{z|_\Gamma}{}{TH(\Gamma)}.
\end{equation}
\end{itemize}
\end{lemma}

\begin{proof}
It is easy to check both items.
\begin{itemize}
\item[(a)] $ \Div\, w $ is well-defined and belongs to $ B^{-1/2}(\bou) $. It remains to prove that $ \mathrm{supp}\, \Div\, w \subset \overline{\Gamma} $. In order to verify this last point, we just need to have in mind the following facts: if $ f \in H^1(\inte) $, then $ \cross{N}{\nabla f} \in TH(\bou) $; moreover $ \mathrm{supp}\, \cross{N}{\nabla f} \subset \mathrm{supp}\, f|_{\bou} $. 
Indeed, $ \mathrm{supp}\, \Div\, w \subset \overline{\Gamma} $ since
\[ \dual{\Div\, w}{f|_{\bou}} = \dual{w}{\cross{N}{(\cross{N}{\nabla f})}}. \]
The estimate is now immediate using either (\ref{es:equivalent_norm_TH-C1}) or (\ref{es:equivalent_norm_TH-C1_forms}).
\item[(b)] By an analogous argument to the one given in (a), we have that if $ \tilde{z}|_\Gamma = z|_\Gamma $ then $ (\Div\, \tilde{z})|_\Gamma = (\Div\, z)|_\Gamma $. Hence the identity follows. The estimate is a consequence of the identity and (\ref{es:equivalent_norm_TH-C1}) or (\ref{es:equivalent_norm_TH-C1_forms}).
\end{itemize}
\end{proof}

\subsection{Maxwell's system as a Sch\"odinger equation}\label{sec:Me-Se}
In this section we shall transform Maxwell's equations into a Schr\"odinger-type equation. The idea of this transformation was already introduced in \cite{OS}.

There is a well known process which allows us to transform Maxwell's equations into a Schr\"odinger-type equation. In order to do so we require some extra smoothness of the coefficients, namely $ \mu, \gamma \in C^{1,1}(\overline{\inte}) $. The first step in this process  is to augment the Maxwell system with two scalar equations:
\begin{equation*}
\delta(\gamma E)=0, \qquad \delta(\mu H)=0.
\end{equation*}
The information coded in these scalar equations was already present in the initial system. In order to check this, it is enough to take $ \ast d $ in each equation in (\ref{ME_without_sources}).

Next, we introduce a new system inspired in the four mentioned equations. This new system reads as
\begin{equation}\label{eq:augmented_sys}
\left\{ 
\begin{array}{l}
\delta(\gamma E) + i \omega \gamma \mu h =0\\
- \gamma^{-1} d(\gamma e) + \ast d E - i \omega \mu H = 0\\
\delta(\mu H) + i \omega \gamma \mu e =0\\
\mu^{-1} d(\mu h) + \ast d H + i \omega \gamma E = 0.\\
\end{array}
\right.
\end{equation}
The new terms preserve the physical units of measure of the original four equations. Choosing euclidean coordinates, the new system --called henceforth \emph{augmented system}-- can be written in vector field notation as it follows
\begin{equation*}
\left[
\left(\begin{array}{c c | c c}
 &  &  & \Diver\\
 &  & D & - \Durl\\
\hline
 & \Diver &  & \\
D & \Durl &  & \\
\end{array}\right)
+\left(\begin{array}{c c | c c}
\omega \mu &  &  & D \alpha \cdot\\
 & \omega \mu I_3 & D \alpha & \\
\hline
 & D \beta \cdot & \omega \gamma & \\
D \beta & & & \omega \gamma I_3\\
\end{array}\right)
\right]\frv{h}{H}{e}{E}=0,
\end{equation*}
where $ \alpha = \mathrm{log} \, \gamma $, $ \beta = \mathrm{log} \, \mu $, $ I_j $ is $ ( j \times j ) $-identity matrix, with $ j \in \mathbb{N} $  and
\begin{gather*}
\Diver = \frac{1}{i}\left(\begin{array}{c c c}
\partial_{x^1} & \partial_{x^2} & \partial_{x^3}
\end{array}\right),\\
D = \frac{1}{i} \left( \begin{array}{c}
\partial_{x^1}\\
\partial_{x^2}\\
\partial_{x^3}
\end{array}\right),\quad
\Durl = \frac{1}{i}\left(\begin{array}{c c c}
 & -\partial_{x^3} & \partial_{x^2} \\
\partial_{x^3} &  & -\partial_{x^1} \\
-\partial_{x^2} & \partial_{x^1} & \\
\end{array}\right).
\end{gather*}
In a much more compact manner we shall express the augmented system as $ ( P + V ) X = 0 $, where
\[ P = \left(\begin{array}{c c | c c}
 &  &  & \Diver\\
 &  & D & - \Durl\\
\hline
 & \Diver &  & \\
D & \Durl &  & \\
\end{array}\right),
\qquad
V = \left(\begin{array}{c c | c c}
\omega \mu &  &  & D \alpha \cdot\\
 & \omega \mu I_3 & D \alpha & \\
\hline
 & D \beta \cdot & \omega \gamma & \\
D \beta & & & \omega \gamma I_3\\
\end{array}\right). \]

Note that $ E, H $ is a solution for Maxwell's equations, if and only if, $ X^t = \fcv{h}{H^t}{e}{E^t} $ is a solution for the augmented system and the scalar fields $ e, h $ vanish.

The next step is to rescale the augmented system, that is
\[ ( P + V ) \, 
\left( \begin{array}{c | c}
\mu^{-1/2} I_4 & \\
\hline
 & \gamma^{-1/2} I_4 \\
\end{array} \right)Y
 = \left( \begin{array}{c | c}
\gamma^{-1/2} I_4 & \\
\hline
 & \mu^{-1/2} I_4 \\
\end{array} \right) \, ( P + W )Y, \]
where
\begin{equation}\label{term:W-def}
W = \kappa I_8 + \frac{1}{2}
\left(\begin{array}{c c|c c}
 &  &  & D \alpha \cdot \\
 &  & D \alpha & D \alpha \times \\
\hline
 & D \beta \cdot &  & \\
D \beta & - D \beta \times &  & \\
\end{array}\right),
\end{equation}
with $ \kappa = \omega \mu^{1/2} \gamma^{1/2}$. We shall call
\[ ( P + W ) Y = 0 \]
the \emph{rescaled system}.

The advantage of rescaling is that
\begin{align}
0 &= ( P + W )( P - W^t ) Z = (-\Delta I_8 + Q) Z, \label{eq:schr1}\\
0 &= ( P - W^t )( P + W ) Z' = (-\Delta I_8 + Q') Z', \label{eq:schr2}\\
0 &= ( P + W^* )( P - \overline{W} ) \hat Z = (-\Delta I_8 + \hat Q) \hat Z, \label{eq:schr3}
\end{align}
where $ Q, Q', \hat Q $ are zeroth-order terms. Here $ W^t $ denotes the transposed of $ W $ and $ W^\ast $ stands for $ \overline{W^t} $. No first order terms appear in (\ref{eq:schr1}), (\ref{eq:schr2}) and (\ref{eq:schr3}), giving as a result a Schr\"odinger-type equation.
Mind
\begin{equation}
Q = -PW^t + WP - WW^t\label{def:matrixQ-0th}.
\end{equation}
Note that if $ Z $ is a solution for (\ref{eq:schr1}) in $ \inte $, then $ Y = ( P - W^t ) Z $ is a solution for the rescaled system in $ \inte $, hence
\[ X = \left( \begin{array}{c | c}
\mu^{-1/2} I_4 &  \\
\hline
 & \gamma^{-1/2} I_4 \\
\end{array} \right) Y \]
is a solution for the augmented system. In the same manner, if $ \hat Z $ is a solution for (\ref{eq:schr3}), then $ \hat Y = (P - \overline{W})\hat Z $ is a solution for $ (P + W^*) \hat Y = 0 $ in $ \inte $.

For later uses,
\begin{gather}
Q = \frac{1}{2} \fbM{\subM{\Delta \alpha}{}{}{2 \nabla^2 \alpha - \Delta \alpha I_3}}{}
{}{\subM{\Delta \beta}{}{}{2 \nabla^2 \beta - \Delta \beta I_3}} + \nonumber
\end{gather}
\begin{gather}
\label{ter:matrixQ-0th}
- \fbM{(\kappa^2 + \frac{1}{4} (\inner{D\alpha}{D\alpha})) I_4}{\subM{}{2 D \kappa \cdot}{2 D \kappa}{}}
{\subM{}{2 D \kappa \cdot}{2 D \kappa}{}}{(\kappa^2 + \frac{1}{4} (\inner{D\beta}{D\beta})) I_4},
\end{gather}
\begin{gather}
Q' = - \frac{1}{2} \fbM{\subM{\Delta \beta}{}{}{2 \nabla^2 \beta - \Delta \beta I_3}}{}
{}{\subM{\Delta \alpha}{}{}{2 \nabla^2 \alpha - \Delta \alpha I_3}}\nonumber \\
\label{ter:matrixQ'-0th}
- \fbM{(\kappa^2 + \frac{1}{4} (\inner{D\beta}{D\beta})) I_4}{\subM{}{}{}{2D \kappa \times}}
{\subM{}{}{}{-2 D \kappa \times}}{(\kappa^2 + \frac{1}{4} (\inner{D\alpha}{D\alpha})) I_4}
\end{gather}
and
\begin{gather}
\hat Q = \frac{1}{2} \fbM{\subM{- \Delta \beta}{}{}{- 2 \nabla^2 \beta + \Delta \beta I_3}}{}
{}{\subM{- \Delta \overline{\alpha}}{}{}{- 2 \nabla^2 \overline{\alpha} + \Delta \overline{\alpha} I_3}}\nonumber \\
\label{ter:matrixQ^-0th}
- \fbM{(\overline{\kappa}^2 - \frac{1}{4} (\inner{D\beta}{D\beta})) I_4}{\subM{}{}{}{-2 D \overline{\kappa} \times}}
{\subM{}{}{}{2 D \overline{\kappa} \times}}{(\overline{\kappa}^2 - \frac{1}{4} (\inner{D \overline{\alpha}}{D \overline{\alpha}})) I_4}
\end{gather}
with $\nabla^2 f=(\partial^2_{x_j, x_k} f)^3_{j,k=1}$.

In order to make as concise as possible the presentation of our proofs, we introduce some additional notation. Let $ Y, Z $ be in the form
\[ Y = \fcv{f^1}{(u^1)^t}{f^2}{(u^2)^t}^t, \qquad Z = \fcv{g^1}{(v^1)^t}{g^2}{(v^2)^t}^t, \]
define
\begin{equation*}
\Inner{Y}{Z} = \sum_{j=1}^{2} \left( \int_{\inte} f^j \overline{g^j} \, dV + \int_{\inte} \inner{u^j}{\overline{v^j}} \, dV \right),
\end{equation*}
\begin{equation*}
\Inner{Y}{Z}_{\bou} =\sum_{j=1}^{2} \left( \int_{\bou} f^j \overline{g^j} \, dA + \int_{\bou} \inner{u^j}{\overline{v^j}} \, dA \right).
\end{equation*}
In the first identity we are assuming $ f^j, g^j \in C^\infty(\overline{\inte}) $ and $ u^j, v^j \in \mathcal{X}\mathbb{E}|_{\overline{\inte}} $ with $ j = 1, 2 $, while in the second identity $ f^j, g^j \in C^\infty(\bou) $ and $ u^j, v^j \in \mathcal{X}\mathbb{E}|_{\bou} $ with $ j = 1, 2 $.
The following integration by parts holds
\[ \Inner{PY}{Z} = \Inner{P_N Y}{Z|_{\bou}}_{\bou} + \Inner{Y}{PZ}. \]
Here, when $ A $ is a (possibly complex) vector field we denote
\begin{equation}\label{for:P_A}
P_A = \frac{1}{i} \left( \begin{array}{c c|c c}
 &  &  & A \cdot \\
 &  & A & - A \times \\
\hline
 & A \cdot &  & \\
A & A \times &  & \\
\end{array}\right).
\end{equation}

Finally, for elements $ Y $ in the form given above we define, for $ |s| > 0 $,
\[ \norm{Y}{}{H^s(\inte; \mathcal{Y})} = \sum_{j = 1, 2} \left( \norm{f^j}{}{H^s(\inte)} + \norm{u^j}{}{H^s(\inte; \mathbb{C}^3)} \right), \]
and
\[ \norm{Y}{}{L^2(\inte; \mathcal{Y})} = \sum_{j = 1, 2} \left( \norm{f^j}{}{L^2(\inte)} + \norm{u^j}{}{L^2(\inte; \mathbb{C}^3)} \right). \]
On the other hand, we define, for $ 0 < |s| < 1 $,
\[ \norm{Y}{}{B^s(\bou; \mathcal{Y})} = \sum_{j = 1, 2} \left( \norm{f^j}{}{B^s(\bou)} + \norm{u^j}{}{B^s(\bou; \mathbb{C}^3)} \right), \]
and
\[ \norm{Y}{}{L^2(\bou; \mathcal{Y})} = \sum_{j = 1, 2} \left( \norm{f^j}{}{L^2(\bou)} + \norm{u^j}{}{L^2(\bou; \mathbb{C}^3)} \right). \]

\subsection{About the geometry of $ U $}\label{sec:geometryU}
In order to make precise the geometrical restrictions assumed in Theorem \ref{th:stability_loc}, Corollary \ref{col:stability_loc} and Theorem \ref{th:uniqueness_loc}, we give the following definitions.

\begin{definition}\label{def:geometryU}\rm
We shall say that a bounded Lipschitz domain $ U \subset \mathbb{E} $ is \emph{partially flat} if there exists a plane $ q \subset \mathbb{E} $ and some euclidean coordinates $ \mathcal{E} $ such that,
\begin{itemize}
\item[(i)] $ q = \{ x \in \mathbb{R}^3 : x^3 = 0 \} $,
\item[(ii)] $ U \subset \{ x \in \mathbb{R}^3 : x^3 < 0 \} $,
\item[(iii)] $ \Gamma_0 := \mathrm{int}_q ( \bouU \cap q ) \neq \emptyset $.
\end{itemize}

We shall say that a bounded Lipschitz domain $ U \subset \mathbb{E} $ is \emph{partially spherical} if there exist a point $ Q_0 \in \mathbb{E} $, $ r_0 > 0 $ and some euclidean coordinates $ \mathcal{E} $ such that
\begin{itemize}
\item[(i)] $ Q_0 = y_0 $ and $ U \subset B(y_0; r_0) := \{ y \in \mathbb{R}^3 : |y - y_0| < r_0 \} $,
\item[(ii)] $ \Gamma_0 := \mathrm{int}_{S(y_0; r_0)} (\bouU \cap S(y_0; r_0)) \neq \emptyset $ where $ S(y_0; r_0) := \partial B(y_0; r_0) $,
\item[(iii)] $ 0 \in S(y_0; r_0) $ but $ 0 \notin \overline{U} $.
\end{itemize}
In the two previous cases, we denote $ \Gamma := \bouU \setminus \overline{\Gamma_0} $.
\end{definition}

\begin{definition}\label{def:suitableU}\rm
We shall say that a partially flat domain $ U $ is \emph{suitable} if its \emph{symmetric extension with respect to $ q $} --that is $ \inte := U \cup \Gamma_0 \cup \mathcal{R}(U) $-- is also Lipschitz. Here $ \mathcal{R} $ denotes the reflection with respect to $ q $ and it is defined as $ (x^1, x^2, x^3) \longmapsto (x^1, x^2, -x^3) $.

In addition, we shall say that a partially spherical domain $ U $ is \emph{suitable} if its \emph{inversion with respect to $ S(0; 2r_0) $} --that is $ \inte := \mathcal{K}(U) $-- is a suitable partially flat domain. Here $ \mathcal{K} $ denotes the inversion with respect to $ S(0; 2r_0) $ and it is defined as $ y \longmapsto r^2_1/|y|^2 y $ with $ r_1 = 2 r_0 $.
\end{definition}

We have to restrict ourselves to these suitable domains because we need to make an extension of the coefficients preserving their smoothness (see Subsection \ref{sec:constructing_sol}).

\section{The domain $ U $ is partially flat}
Along this section we assume $ U $ to be a suitable partially flat domain and we follow the notation in Definition \ref{def:geometryU} and Definition \ref{def:suitableU}.

\subsection{Maxwell's system and the reflection map}\label{sec:reflec_map}
Let the coefficients $ \mu, \gamma $ be such that $ \mu, \gamma \in C^{1,1}(\overline{U}) $ with $ \partial_{x^3} \mu |_{\Gamma_0}= \partial_{x^3} \gamma |_{\Gamma_0} = 0 $ and set $ \tilde{\mu}, \tilde{\gamma} : \overline{\inte} \longrightarrow \mathbb{C} $ two smooth extensions of $ \mu $ and $ \gamma $ defined as
\[ \tilde{\mu}(x^1, x^2, x^3) = \mu(x^1, x^2, -|x^3|), \qquad \tilde{\gamma}(x^1, x^2, x^3) = \gamma(x^1, x^2, -|x^3|), \]
for any $ x \in \overline{\inte} $. Note that the hypothesis $ \partial_{x^3} \mu |_{\Gamma_0}= \partial_{x^3} \gamma |_{\Gamma_0} = 0 $ allows us to keep the smoothness when extending.

Consider the system
\begin{equation}\label{eq:ME_reflected_domain}
\left\{ 
\begin{array}{l}
\curl H + i \omega \tilde{\gamma} E = 0\\
\curl E - i \omega \tilde{\mu} H = 0\\
\end{array}
\right.
\end{equation}
in $ \inte $.
The push-forward of the reflection map $ \mathcal{R} $ reads
\begin{equation*}
\mathcal{R}_* =
\left(\begin{array}{c c c}
\, 1 &  & \\
 & \, 1 & \\
 &  & -1 \\
\end{array}\right).
\end{equation*}

Let $ f $ be a smooth function and $ u, u' $ two vector fields on $ \mathbb{E} $. Let $ g, v, v' $ denote the function and the vector fields given by
\[ g(x) := f(\mathcal{R}(x)), \qquad v_x := \mathcal{R}_* u_{\mathcal{R}(x)}, \qquad v'_x := \mathcal{R}_* u_{\mathcal{R}(x)}. \]
It is a straight forward computation to check that
\begin{align}
(\diver v)(x) &= (\diver u)(\mathcal{R}(x)), \label{id:reflection-diver} \\
(\nabla g)_x &= \mathcal{R}_* (\nabla f)_{\mathcal{R}(x)}, \label{id:reflection-grad} \\
(\curl v)_x &= - \mathcal{R}_* (\curl u)_{\mathcal{R}(x)}, \label{id:reflection-curl} \\
(\cross{v}{v'})_x &= - \mathcal{R}_*(\cross{u}{u'})_{\mathcal{R}(x)}. \label{id:reflection-cross}
\end{align}

On the other hand, let $ a $ be a smooth function defined in $ \{ x \in \mathbb{R}^3 : x^3 < 0 \} $ and set $ \tilde{a} $, the extension of $ a $ to $ \mathbb{E} $, defined as $ \tilde{a} (x^1, x^2, x^3) = a (x^1, x^2, -|x^3|) $. Then
\begin{equation}\label{id:reflection-coeffgrad}
(\nabla \tilde{a})_x = \mathcal{R}_* (\nabla \tilde{a})_{\mathcal{R}(x)}.
\end{equation}

\begin{lemma}\label{le:vanishing_bou_1}\rm
Given
\[ Y = \fcv{0}{\tilde{\mu}^{1/2} H^t}{0}{\tilde{\gamma}^{1/2} E^t}^t, \]
such that $ E, H \in \Hcurl{} $ is a solution of (\ref{eq:ME_reflected_domain}) in $ \inte $, one has that $ E - \dot{E}, H - \dot{H} $, with
\[ \dot{E}_x := \mathcal{R}_* E_{\mathcal{R}(x)}, \qquad \dot{H}_x := - \mathcal{R}_* H_{\mathcal{R}(x)}; \]
is also a solution of (\ref{eq:ME_reflected_domain}) in $ \inte $ satisfying
\begin{equation}\label{ter:vanishing_atboundary_1}
\cross{N}{(E - \dot{E})}|_{\Gamma_0} = 0.
\end{equation}
\end{lemma}

\begin{proof}
Let $ E, H $ be a solution of (\ref{eq:ME_reflected_domain}) in $ \inte $. It is an immediate consequence of (\ref{id:reflection-curl}) and the definition of $ \tilde{\mu}, \tilde{\gamma} $ in $ \inte $ that $ \dot{E}, \dot{H} $ is also a solution for (\ref{eq:ME_reflected_domain}) in $ \inte $. Further, from the weak definition of tangential trace one can derive that $ \cross{N}{(E - \dot{E})}|_{\Gamma_0} = 0 $. Indeed, let $ w \in B^{1/2}(\bouU; \mathbb{C}^3) $ such that $ \mathrm{supp}\, w \subset \overline{\Gamma_0} $ and consider $ v \in H^{1}(U; \mathbb{C}^3) $ such that $ v|_{\bouU} = w $, then
\begin{gather*}
\dual{\cross{N}{E}-\cross{N}{\dot E}}{w}_{\bouU} = \int_{U} \inner{(\curl{E} - \curl{\dot E})}{\overline{v}} \,dV - \int_{U} \inner{(E - \dot E)}{\overline{\curl v}} \,dV\\
= \int_{U} \inner{\curl{E}}{\overline{v}} - \inner{E}{\overline{\curl v}} \,dV
  + \int_{U} \inner{(\curl{E})_{\mathcal{R}}}{\overline{\mathcal{R}_* v}} + \inner{E_{\mathcal{R}}}{\overline{\mathcal{R}_* \curl v}} \,dV\\
= \int_{U} \inner{\curl{E}}{\overline{v}} - \inner{E}{\overline{\curl v}} \,dV
  + \int_{\mathcal{R}(U)} \inner{(\curl{E})}{\overline{\mathcal{R}_* v_{\mathcal{R}}}} - \inner{E}{\overline{(\curl \mathcal{R}_* v_{\mathcal{R}})}} \,dV\\
= \int_{\inte} \inner{\curl{E}}{\overline{u}} - \inner{E}{\overline{\curl u}} \,dV = \dual{\cross{N}{E}}{u}_{\bou} = 0.
\end{gather*}
Here we have used (\ref{id:reflection-curl}) twice, and the fact that $ u $, defined as $ v $ in $ U $ and as $ \mathcal{R}_* v_{\mathcal{R}} $ in $ \mathcal{R}(U) $, belongs to $ H^1(\inte; \mathbb{C}^3) $ and $ u|_{\bou} = 0 $.
\end{proof}

\begin{lemma}\label{le:vanishing_bou_2}\rm
Given
\[ Y = \fcv{f^1}{u^1}{f^2}{u^2} \]
solution of $ (P + W^*)Y = 0 $ in $ \inte $ with $ f^j \in H^1(\inte) $ and $ u^j \in \Hcurl{} $, one has that $ Y - \dot{Y} $ is also a solution of $ (P + W^*) (Y - \dot{Y}) = 0 $ in $ \inte $. Here $ W $ denotes the matrix (\ref{term:W-def}) for coefficients $ \tilde{\mu}, \tilde{\gamma} $ and $ \dot{Y}_x := \dot{J} Y_{\mathcal{R}(x)} $
with
\[Y_x = \fcv{f^1(x)}{u^1_x}{f^2(x)}{u^2_x} \qquad \dot{J} := \left( \begin{array}{c c| c c}
1 &  &  &  \\
 & -\mathcal{R}_* &  &  \\
\hline
 &  & -1 &  \\
  &  &  & \mathcal{R}_* 
\end{array} \right).\]
Additionally,
\begin{equation}\label{ter:vanishing_atboundary_2}
(f^1 - \dot{f}^1)|_{\Gamma_0} = 0, \qquad \cross{N}{(u^2 - \dot{u}^2)}|_{\Gamma_0} = 0.
\end{equation}
\end{lemma}

\begin{proof}
The first part of the lemma follows from
\begin{equation}\label{id:P-refle}
(P \dot{Y})_x = \dot{J} (P Y)_{\mathcal{R}(x)}
\end{equation}
and
\begin{equation}\label{id:W*-refle}
(W^* \dot{Y})_x = \dot{J}(W^* Y)_{\mathcal{R}(x)}.
\end{equation}
The identity (\ref{id:P-refle}) is a consequence of (\ref{id:reflection-diver}), (\ref{id:reflection-grad}) and (\ref{id:reflection-curl}). The identity (\ref{id:W*-refle}) follows from (\ref{id:reflection-coeffgrad}) and (\ref{id:reflection-cross}).

Additionally, $ (f^1 - \dot{f}^1)|_{\Gamma_0} = 0 $ since $ f^1 \in H^1(\inte) $ and $ \cross{N}{(u^2 - \dot{u}^2)}|_{\Gamma_0} = 0 $ as we showed in the proof of Lemma \ref{le:vanishing_bou_1}.
\end{proof}

\subsection{Relating the boundary measurements with the coefficients in the interior}

\begin{lemma}\label{le:1-suitable-estimate_loc}\rm
Let $ \mu_j, \gamma_j $ belong to $ C^{0,1}(\overline{U}) $. Then, for any $ Y_1 $ given as in the hypothesis of Lemma \ref{le:vanishing_bou_1} with coefficients $ \tilde{\mu}_1, \tilde{\gamma}_1 $ and any $ Y_2 $ given as in the hypothesis of Lemma \ref{le:vanishing_bou_2} with coefficients $ \tilde{\mu}_2, \tilde{\gamma}_2 $, one has that the elements $ \mathpzc{E}_1 = E_1 - \dot{E}_1 $, $ \mathpzc{H}_1 = H_1 - \dot{H}_1 $ and $ \mathpzc{Y}_2 = Y_2 - \dot{Y}_2 $, expressed in the form
\[ Y_1 = \fcv{0}{\tilde{\mu}_1^{1/2} H_1^ t}{0}{\tilde{\gamma}_1^{1/2} E_1^t}^t
\quad \mathpzc{Y}_2 = \fcv{\mathpzc{f}^1}{(\mathpzc{u}^1)^t}{\mathpzc{f}^2}{(\mathpzc{u}^2)^t}^t, \]
satisfy the following estimate
\begin{gather*}
|\Inner{Y_1}{P\mathpzc{Y}_2}_{\inte} - \Inner{PY_1}{\mathpzc{Y}_2}_{\inte}| \leq\\ 
\leq C \delta_C(C^1_\Gamma, C^2_\Gamma) \left( \norm{\mu_2^{-1/2}}{}{C^{0,1}(\overline{\Gamma})} \norm{\mathpzc{g}_2|_\Gamma}{}{B^{1/2}(\Gamma)} + \norm{\gamma_2^{1/2}}{}{C^{0,1}(\overline{\Gamma})} \norm{\mathpzc{z}_1|_\Gamma}{}{TH(\Gamma)} + \right.
\end{gather*}
\begin{gather*}
\left. + \norm{\gamma_2^{-1/2}}{}{C^{0,1}(\overline{\Gamma})} \norm{\mathpzc{g}_1}{}{B^{1/2}_0(\Gamma)}
+ \norm{\mu_2^{1/2}}{}{C^{0,1}(\overline{\Gamma})} \norm{\mathpzc{z}_2}{}{TH_0(\Gamma)} \right) \norm{\cross{N}{\mathpzc{E}_1}}{}{TH_0(\Gamma)}\\
+ C \left( \norm{\mu_1^{-1/2} - \mu_2^{-1/2}}{}{C^{0,1}(\overline{\Gamma})} \norm{\mathpzc{g}_2|_\Gamma}{}{B^{1/2}(\Gamma)} \right. \\
+ \norm{\gamma_1^{1/2} - \gamma_2^{1/2}}{}{C^{0,1}(\overline{\Gamma})} \norm{\mathpzc{z}_1|_\Gamma}{}{TH(\Gamma)} + \norm{\gamma_1^{-1/2} - \gamma_2^{-1/2}}{}{C^{0,1}(\overline{\Gamma})} \norm{\mathpzc{g}_1}{}{B^{1/2}_0(\Gamma)}\\
\left. + \norm{\mu_1^{1/2} - \mu_2^{1/2}}{}{C^{0,1}(\overline{\Gamma})} \norm{\mathpzc{z}_2}{}{TH_0(\Gamma)} \right)\left( \norm{\cross{N}{\mathpzc{E}_1}}{}{TH_0(\Gamma)} + \norm{\cross{N}{\mathpzc{H}_1}|_\Gamma}{}{TH(\Gamma)} \right).
\end{gather*}
Here $ \mathpzc{g}_1, \mathpzc{g}_2 \in B^{1/2}(\bouU) $ stand for $ \mathpzc{g}_1 = \mathpzc{f}^1|_{\bouU}, \, \mathpzc{g}_2 = \mathpzc{f}^2|_{\bouU} $ and $ \mathpzc{z}_1, \mathpzc{z}_2 \in TH(\bouU) $ stand for $ \mathpzc{z}_1 = \cross{N}{\mathpzc{u}^1}, \, \mathpzc{z}_2 = \cross{N}{\mathpzc{u}^2} $. Here $ C^j_\Gamma = C(\mu_j, \gamma_j; \Gamma) $ with $ j=1,2 $.
\end{lemma}

\begin{proof}
It is easy to check, using (\ref{id:P-refle}) that
\[ \Inner{Y_1}{P\mathpzc{Y}_2}_{\inte} - \Inner{PY_1}{\mathpzc{Y}_2}_{\inte} = \Inner{\mathpzc{Y}_1}{P\mathpzc{Y}_2}_{U} - \Inner{P\mathpzc{Y}_1}{\mathpzc{Y}_2}_{U}, \]
where $ \mathpzc{Y}_1 = Y_1 - \dot{Y}_1 $. Let $ \mathpzc{L} $ be
\[ \mathpzc{L} = \fcv{0}{\mu_2^{1/2} \mathpzc{H}_2^t}{0}{\gamma_2^{1/2} \mathpzc{E}_2^t}^t, \]
with $ \mathpzc{E}_2, \mathpzc{H}_2 \in \HcurlU{} $ an arbitrary solution of
\begin{equation}\label{eq:ME_intermSOL}
\curl \mathpzc{H}_2 + i\omega \gamma_2 \mathpzc{E}_2 = 0, \qquad \curl \mathpzc{E}_2 - i\omega \mu_2 \mathpzc{H}_2 = 0
\end{equation}
in $ U $ and satisfying $ \mathrm{supp}\, \cross{N}{\mathpzc{E}_2} \subset \overline{\Gamma} $. Since $ ( P + W_2^\ast ) \mathpzc{Y}_2 = 0 $ and $ ( P + W_2 ) \mathpzc{L} = 0 $ in $ U $, one has that $ \Inner{\mathpzc{L}}{P\mathpzc{Y}_2}_U = \Inner{P\mathpzc{L}}{\mathpzc{Y}_2}_U $, hence
\[ \Inner{\mathpzc{Y}_1}{P\mathpzc{Y}_2}_U - \Inner{P\mathpzc{Y}_1}{\mathpzc{Y}_2}_U = \Inner{\mathpzc{Y}_1 - \mathpzc{L}}{P\mathpzc{Y}_2}_U - \Inner{P(\mathpzc{Y}_1 - \mathpzc{L})}{\mathpzc{Y}_2}_U. \]
On the other hand, we have, using (\ref{def:normal-component}) 
and (\ref{for:duality_TH-TH*}) 
, that
\begin{gather*}
\Inner{\mathpzc{Y}_1 - \mathpzc{L}}{P\mathpzc{Y}_2}_U - \Inner{P(\mathpzc{Y}_1 - \mathpzc{L})}{\mathpzc{Y}_2}_U =\\
= i\dual{\inner{N}{(\mu_1 \mathpzc{H}_1 - \mu_2 \mathpzc{H}_2)}}{\mu_2^{-1/2} \mathpzc{g}_2} + i\dual{\inner{N}{(\mu_1 \mathpzc{H}_1)}}{(\mu_1^{-1/2} - \mu_2^{-1/2}) \mathpzc{g}_2}\\
+ i\dual{\inner{N}{(\gamma_1 \mathpzc{E}_1 - \gamma_2 \mathpzc{E}_2)}}{\overline{\gamma_2^{-1/2}} \mathpzc{g}_1} + i\dual{\inner{N}{(\gamma_1 \mathpzc{E}_1)}}{(\overline{\gamma_1^{-1/2} - \gamma_2^{-1/2}}) \mathpzc{g}_1}\\
- i\dual{\cross{N}{(\mathpzc{H}_1-\mathpzc{H}_2)}}{\cross{N}{(\mu_2^{1/2}\mathpzc{z}_2)}} - i\dual{\cross{N}{\mathpzc{H}_1}}{\cross{N}{((\mu_1^{1/2} - \mu_2^{1/2})\mathpzc{z}_2)}}\\
- i\dual{\cross{N}{(\gamma_2^{1/2} \cross{N}{(\mathpzc{E}_1-\mathpzc{E}_2)})}}{\mathpzc{z}_1} - i\dual{\cross{N}{((\gamma_1^{1/2} - \gamma_2^{1/2})\cross{N}{\mathpzc{E}_1})}}{\mathpzc{z}_1}.
\end{gather*}
Furthermore, from the Maxwell's equations one deduces that
\[ \inner{N}{(\gamma_j \mathpzc{E}_j)} = \frac{1}{i\omega} \Div \, (\cross{N}{\mathpzc{H}_j}), \quad \inner{N}{(\mu_j \mathpzc{H}_j)} = -\frac{1}{i\omega} \Div \, (\cross{N}{\mathpzc{E}_j}), \]
for $ j = 1,2 $. Hence,
\begin{gather*}
\Inner{\mathpzc{Y}_1}{P\mathpzc{Y}_2}_U - \Inner{P\mathpzc{Y}_1}{\mathpzc{Y}_2}_U = \\
= -\frac{1}{\omega} \dual{\Div (\cross{N}{\mathpzc{E}_1} - \cross{N}{\mathpzc{E}_2})}{\mu_2^{-1/2} \mathpzc{g}_2} -\frac{1}{\omega} \dual{\Div\, \cross{N}{\mathpzc{E}_1}}{(\mu_1^{-1/2} - \mu_2^{-1/2}) \mathpzc{g}_2}\\
+ \frac{1}{\omega} \dual{\Div(\cross{N}{\mathpzc{H}_1} - \cross{N}{\mathpzc{H}_2})}{\overline{\gamma_2^{-1/2}} \mathpzc{g}_1} + \frac{1}{\omega} \dual{\Div\, \cross{N}{\mathpzc{H}_1}}{(\overline{\gamma_1^{-1/2} - \gamma_2^{-1/2}}) \mathpzc{g}_1}\\
- i\dual{\cross{N}{\mathpzc{H}_1} - \cross{N}{\mathpzc{H}_2}}{\cross{N}{(\mu_2^{1/2}\mathpzc{z}_2)}} - i\dual{\cross{N}{\mathpzc{H}_1}}{\cross{N}{((\mu_1^{1/2} - \mu_2^{1/2})\mathpzc{z}_2)}}\\
- i\dual{\cross{N}{(\gamma_2^{1/2} (\cross{N}{\mathpzc{E}_1} - \cross{N}{\mathpzc{E}_2}))}}{\mathpzc{z}_1} - i\dual{\cross{N}{((\gamma_1^{1/2} - \gamma_2^{1/2}) \cross{N}{\mathpzc{E}_1})}}{\mathpzc{z}_1}.
\end{gather*}
Let us denote $ \cross{N}{\mathpzc{E}_j} = T_j $, $ \cross{N}{\mathpzc{H}_j}|_\Gamma = S_j $, then by using the appropriate dualities, the boundary conditions (\ref{ter:vanishing_atboundary_1}), (\ref{ter:vanishing_atboundary_2}) and the estimates (\ref{es:besov_product_0}), (\ref{es:besov_product_Gamma}), (\ref{es:TH-product_0}), (\ref{es:Div-boun-TH0}) and (\ref{es:Div-boun-THGamma}) we get
\begin{gather*}
|\Inner{Y_1}{P\mathpzc{Y}_2}_{\inte} - \Inner{PY_1}{\mathpzc{Y}_2}_{\inte}| \leq\\ 
\leq C \left( \norm{T_1 - T_2}{}{TH_0(\Gamma)} + \norm{S_1 - S_2}{}{TH(\Gamma)} \right) \left( \norm{\mu_2^{-1/2}}{}{C^{0,1}(\overline{\Gamma})} \norm{\mathpzc{g}_2|_\Gamma}{}{B^{1/2}(\Gamma)} \right. \\
+ \norm{\gamma_2^{1/2}}{}{C^{0,1}(\overline{\Gamma})} \norm{\mathpzc{z}_1|_\Gamma}{}{TH(\Gamma)} + \norm{\gamma_2^{-1/2}}{}{C^{0,1}(\overline{\Gamma})} \norm{\mathpzc{g}_1}{}{B^{1/2}_0(\Gamma)}\\
\left. + \norm{\mu_2^{1/2}}{}{C^{0,1}(\overline{\Gamma})} \norm{\mathpzc{z}_2}{}{TH_0(\Gamma)} \right) + C \left( \norm{\mu_1^{-1/2} - \mu_2^{-1/2}}{}{C^{0,1}(\overline{\Gamma})} \norm{\mathpzc{g}_2|_\Gamma}{}{B^{1/2}(\Gamma)} \right. \\
+ \norm{\gamma_1^{1/2} - \gamma_2^{1/2}}{}{C^{0,1}(\overline{\Gamma})} \norm{\mathpzc{z}_1|_\Gamma}{}{TH(\Gamma)} + \norm{\gamma_1^{-1/2} - \gamma_2^{-1/2}}{}{C^{0,1}(\overline{\Gamma})} \norm{\mathpzc{g}_1}{}{B^{1/2}_0(\Gamma)}\\
\left. + \norm{\mu_1^{1/2} - \mu_2^{1/2}}{}{C^{0,1}(\overline{\Gamma})} \norm{\mathpzc{z}_2}{}{TH_0(\Gamma)} \right)\left( \norm{T_1}{}{TH_0(\Gamma)} + \norm{S_1}{}{TH(\Gamma)} \right).
\end{gather*}
This estimate holds for all $ (T_2, S_2) \in C^2_\Gamma $, since $ \mathpzc{E}_2, \mathpzc{H}_2 $ was chosen to be an arbitrary solution of (\ref{eq:ME_intermSOL}) in $ U $ satisfying $ \mathrm{supp}\, \cross{N}{\mathpzc{E}_2} \subset \overline{\Gamma} $. Finally, the wanted estimate is a consequence of Definition \ref{def:Cdistance_res}.
\end{proof}

\begin{proposition}\label{le:2-suitable-estimate_loc}\rm
Let $ \gamma_1, \mu_1 $ and $ \gamma_2, \mu_2 $ be in the class of $ B $-stable coefficients on $ \Gamma $ at frequency $ \omega $. Then, there exists a constant $ C(M) $ such that, for any $ Z_1 \in H^1(\inte; \mathcal{Y}) $ satisfying $ Y_1 = (P - W_1^t) Z_1 $ with $ Y_1 $ as in Lemma \ref{le:1-suitable-estimate_loc} and any $ \mathpzc{Y}_2 \in H^1(\inte; \mathcal{Y}) $ as in Lemma \ref{le:1-suitable-estimate_loc}, one has
\begin{gather*}
|\Inner{(Q_1 - Q_2)Z_1}{\mathpzc{Y}_2}_{\inte}| \leq C\, B\big( \delta_C(C^1_\Gamma, C^2_\Gamma) \big) \norm{Z_1}{}{H^1(\inte; \mathcal{Y})} \norm{\mathpzc{Y}_2}{}{H^1(\inte; \mathcal{Y})}\\
+\, C \, B\big( \delta_C(C^1_\Gamma, C^2_\Gamma) \big) \left( \norm{\mathpzc{E}_1}{}{\HcurlU{}} + \norm{\mathpzc{H}_1}{}{\HcurlU{}} \right)\\
\times \left( \norm{\mathpzc{f}^1}{}{H^1(U)} + \norm{\mathpzc{u}^1}{}{\HcurlU{}} + \norm{\mathpzc{f}^2}{}{H^1(U)} + \norm{\mathpzc{u}^2}{}{\HcurlU{}} \right),
\end{gather*}
Here $ Q_j $ is the matrix (\ref{def:matrixQ-0th}) associates to $ \tilde{\mu}_j, \tilde{\gamma}_j $ with $ j=1,2 $.
\end{proposition}

\begin{proof}
From (\ref{def:matrixQ-0th}) one has
\begin{gather*}
\Inner{(Q_1 - Q_2)Z_1}{\mathpzc{Y}_2}_{\inte} = - \Inner{P(W^t_1 - W^t_2)Z_1}{\mathpzc{Y}_2}_{\inte}\\
+ \Inner{(W_1 - W_2)PZ_1}{\mathpzc{Y}_2}_{\inte} - \Inner{(W_1W^t_1 - W_2W^t_2)Z_1}{\mathpzc{Y}_2}_{\inte}\\
= \Inner{(W_1^t - W_2^t)Z_1}{P_N \mathpzc{Y}_2}_{\bou} - \Inner{W^t_1 Z_1}{P\mathpzc{Y}_2}_{\inte} + \Inner{W^t_2 Z_1}{P\mathpzc{Y}_2}_{\inte}\\
  + \Inner{W_1(P - W_1^t) Z_1}{\mathpzc{Y}_2}_{\inte} - \Inner{PZ_1}{W_2^\ast \mathpzc{Y}_2}_{\inte} + \Inner{W^t_2 Z_1}{W_2^\ast \mathpzc{Y}_2}_{\inte}\\
= \Inner{(W_1^t - W_2^t)Z_1}{P_N \mathpzc{Y}_2}_{\bou} + \Inner{(P-W_1^t)Z_1}{P\mathpzc{Y}_2}_{\inte}  + \Inner{W_1(P-W_1^t)Z_1}{\mathpzc{Y}_2}_{\inte}\\
= \Inner{(W_1^t - W_2^t)Z_1}{P_N \mathpzc{Y}_2}_{\bou} + \Inner{Y_1}{P\mathpzc{Y}_2}_{\inte} - \Inner{PY_1}{\mathpzc{Y}_2}_{\inte}.
\end{gather*}
In order to get the penultimate identity, we used twice that $ (P + W_2^\ast)\mathpzc{Y}_2 = 0 $, while to get the last one, we used that $ Y_1 = (P-W^t_1)Z_1 $ and that $ (P + W_1)Y_1 = 0 $.

It is a straight forward computation to check the next estimate
\begin{gather*}
|\Inner{(W_1^t - W_2^t)Z_1}{P_N Y_2}_{\bou}| \leq C \left( \norm{\kappa_1 - \kappa_2}{}{L^\infty(\Gamma)} + \norm{\nabla(\beta_1 - \beta_2)}{}{L^\infty(\Gamma; \mathbb{C}^3)} \right.\\
\left. + \norm{\nabla(\alpha_1 - \alpha_2)}{}{L^\infty(\Gamma; \mathbb{C}^3)} \right) \norm{Z_1}{}{L^2(\bou; \mathcal{Y})} \norm{\mathpzc{Y}_2}{}{L^2(\bou; \mathcal{Y})}.
\end{gather*}
Here, as usually, the norm of $ L^\infty (\Gamma; \mathbb{C}^3) $ is
\[ \norm{w}{2}{L^\infty(\Gamma; \mathbb{C}^3)} = \sum_{j=1}^3 \norm{w^{(j)}}{2}{L^\infty(\Gamma)}, \]
for any vector field $ w $.
It is a routine computation to check that, on one hand
\begin{gather*}
\norm{\kappa_1 - \kappa_2}{}{L^\infty(\Gamma)} \leq C\, B\big( \delta_C(C^1_\Gamma, C^2_\Gamma) \big),\\
\norm{\nabla(\alpha_1 - \alpha_2)}{}{L^\infty(\Gamma; \mathbb{C}^3)} \leq C\, B\big( \delta_C(C^1_\Gamma, C^2_\Gamma) \big),\\
\norm{\nabla(\beta_1 - \beta_2)}{}{L^\infty(\Gamma; \mathbb{C}^3)} \leq C\, B\big( \delta_C(C^1_\Gamma, C^2_\Gamma) \big).
\end{gather*}
and on the other hand,
\begin{gather*}
\norm{\mu_2^{-1/2}}{}{C^{0,1}(\overline{\Gamma})} + \norm{\gamma_2^{-1/2}}{}{C^{0,1}(\overline{\Gamma})} + \norm{\mu_2^{1/2}}{}{C^{0,1}(\overline{\Gamma})} + \norm{\gamma_2^{1/2}}{}{C^{0,1}(\overline{\Gamma})} \leq C\\
\norm{\mu_1^{-1/2} - \mu_2^{-1/2}}{}{C^{0,1}(\overline{\Gamma})} + \norm{\mu_1^{1/2} - \mu_2^{1/2}}{}{C^{0,1}(\overline{\Gamma})} \leq C\, B\big( \delta_C(C^1_\Gamma, C^2_\Gamma) \big),\\
\norm{\gamma_1^{-1/2} - \gamma_2^{-1/2}}{}{C^{0,1}(\overline{\Gamma})} + \norm{\gamma_1^{1/2} - \gamma_2^{1/2}}{}{C^{0,1}(\overline{\Gamma})} \leq C\, B\big( \delta_C(C^1_\Gamma, C^2_\Gamma) \big),
\end{gather*}
With all these estimates and Lemma \ref{le:1-suitable-estimate_loc} in mind, we get
\begin{gather*}
|\Inner{(Q_1 - Q_2)Z_1}{\mathpzc{Y}_2}_{\inte}| \leq C\, B\big( \delta_C(C^1_\Gamma, C^2_\Gamma) \big) \norm{Z_1}{}{B^{1/2}(\bou; \mathcal{Y})} \norm{\mathpzc{Y}_2}{}{B^{1/2}(\bou; \mathcal{Y})} +
\end{gather*}
\begin{gather*}
+\, C \, B\big( \delta_C(C^1_\Gamma, C^2_\Gamma) \big) \left( \norm{\cross{N}{\mathpzc{E}_1}}{}{TH_0(\Gamma)} + \norm{\cross{N}{\mathpzc{H}_1}|_\Gamma}{}{TH(\Gamma)} \right)\\
\times \left( \norm{\mathpzc{g}_1}{}{B^{1/2}_0(\Gamma)} + \norm{\mathpzc{z}_1|_\Gamma}{}{TH(\Gamma)} + \norm{\mathpzc{g}_2|_\Gamma}{}{B^{1/2}(\Gamma)} + \norm{\mathpzc{z}_2}{}{TH_0(\Gamma)} \right),
\end{gather*}
hence we deduce the estimate given in the statement.
\end{proof}

\subsection{Recalling the existence of special solutions}\label{sec:constructing_sol}
Let $ B(O; \rho) $ be the open ball centered at the origin $ O $ with radius $ \rho > 0 $ and such that $ \overline{\inte} \subset B(O; \rho) $. Sometimes $ B(O; \rho) $ will be denoted by $ B $ to simplify the notation. Let $ \varepsilon_0 $ and $ \mu_0 $ denote the electric and magnetic constants, respectively. Extend the coefficients $ \tilde{\mu}_j, \tilde{\gamma}_j $ defined in $ \inte $ to functions in $ \mathbb{E} $ --still denoted by $ \tilde{\mu}_j, \tilde{\gamma}_j $--, preserving their smoothness and in such a way that $ \tilde{\mu}_j - \mu_0, \tilde{\gamma}_j - \varepsilon_0 $ have compact support in $ \overline{B(O;\rho)} $ (regarding to extensions see \cite{St}). Note two simple facts. Firstly, the extensions still satisfy the a priori bound and the a priori ellipticity condition in $ \mathbb{E} $. Secondly, the extensions of the matrices (\ref{ter:matrixQ-0th}), (\ref{ter:matrixQ'-0th}) (\ref{ter:matrixQ^-0th}) --still denoted by $ Q_j, Q'_j, \hat{Q}_j $-- satisfy that $ \omega^2 \varepsilon_0 \mu_0 I_8 + Q_j $, $ \omega^2 \varepsilon_0 \mu_0 I_8 + Q'_j $ and $ \omega^2 \varepsilon_0 \mu_0 I_8 + \hat{Q}_j $ have compact support in $ \overline{B(O;\rho)} $.

In the following, we state two propositions which were proven in \cite{C}. Their proofs are based on ideas from \cite{SyU}, \cite{B}, \cite{OS} and \cite{KSaU}.

\begin{proposition}\label{pro:CGO-sch} \rm
Let $ \delta $ be a constant such that $ -1 < \delta < 0 $ and let $ \zeta \in \mathbb{C}^3 $ be such that $ \zeta \cdot \zeta = \omega^2 \varepsilon_0 \mu_0 $ with
\[ |\zeta| > C(\delta, \rho)\left(\sum_{j = 1,2} \norm{\omega^2 \varepsilon_0 \mu_0 + q_j}{}{L^\infty(B)} + \sum_{j,k=1}^8 \norm{(\omega^2 \varepsilon_0 \mu_0 I_8 + Q)_j^k}{}{L^\infty(B)} \right), \]
where
\begin{equation*}
q_1 = -\frac{1}{2} \Delta \beta - \kappa^2 - \frac{1}{4} (\inner{D\beta}{D\beta}), \quad q_2 = - \frac{1}{2} \Delta \alpha - \kappa^2 - \frac{1}{4} (\inner{D\alpha}{D\alpha}).
\end{equation*}
Then, there exists a
\[ Z = e^{i \zeta \cdot x} (L + R) \]
solution of $ (-\Delta I_8 + Q) Z = 0 $ in $ \mathbb{E} $, with $ Z|_\inte \in H^2(\inte; \mathcal{Y}) $,
\[ L = \frac{1}{|\zeta|} \frv{\zeta \cdot A}{\omega \varepsilon_0^{1/2} \mu_0^{1/2} B}{\zeta \cdot B}{\omega \varepsilon_0^{1/2} \mu_0^{1/2} A}, \]
for $ A, B $ constant complex vector fields, and $ R $ satisfying
\begin{align*}
\norm{R}{}{L^2_\delta \mathcal{Y}} &\leq \frac{C(\delta, \rho)}{|\zeta|} |L| \sum_{j,k=1}^8 \norm{(\omega^2 \varepsilon_0 \mu_0 I_8 + Q)_j^k}{}{L^\infty(B)}.
\end{align*}
Furthermore, $ Y = (P - W^t)Z $ is solution for $ (P + W)Y = 0 $ in $ \mathbb{E} $ and it reads
\[ Y = \fcv{0}{\mu^{1/2} H^t}{0}{\gamma^{1/2} E^t}^t \]
with $ E, H $ solution for (\ref{ME_without_sources}) in $ \mathbb{E} $.
\end{proposition}

The norm in the proposition is
\[ \norm{f}{2}{L^2_\delta} = \int_{\mathbb{R}^3} (1 + |x|^2)^\delta |f|^2 \, dx. \]
Again $ \mathcal{Y} $ is meanless, it just stands to remark the form of the elements for which the norms are taken.

\begin{proposition}\label{pro:CGO-rescale*}\rm
Let $ \zeta \in \mathbb{C}^3 $ be such that $ \zeta \cdot \zeta = \omega^2 \varepsilon_0 \mu_0 $ with
\[ |\zeta| > C(\rho) \sum_{j,k=1}^8 \norm{(\omega^2 \varepsilon_0 \mu_0 I_8 + \hat Q)_j^k}{}{L^\infty(B)} .\]
Then, there exists a
\[ \hat Y = e^{i\zeta \cdot x}(M + S) \]
solution for the equation $ (P + W^*) \hat Y = 0 $ in $ \mathbb{E} $, with $ \hat Y|_\inte \in H^1(\inte; \mathcal{Y}) $,
\[ M = \frac{1}{|\zeta|}\frv{\inner{\zeta}{\hat A}}{-\cross{\zeta}{\hat A}}{\inner{\zeta}{\hat B}}{\cross{\zeta}{\hat B}}, \]
for $ \hat A, \hat B $ constant complex vector fields, and $ S $ satisfying
\begin{equation*}
\norm{S}{}{L^2(\inte; \mathcal{Y})} \leq \frac{C(\rho, \inte)}{|\zeta|} \sum_{j,k=1}^8 \left( \norm{(\omega^2 \varepsilon_0 \mu_0 I_8 + \hat Q)_j^k}{}{L^\infty(B)} + \norm{W_j^k}{}{L^\infty(\inte)} \right).
\end{equation*}
\end{proposition}

\subsection{Proof of the stability}
The main idea in this final part goes back to \cite{A}.

Let $ \tilde{\mu}_1, \tilde{\gamma}_1 $ and $ \tilde{\mu}_2, \tilde{\gamma}_2 $ be two pairs of coefficients as in Theorem \ref{th:stability_loc} and choose
\begin{gather}
\zeta_1=-\frac{1}{2}\xi+i\left(\tau^2+\frac{|\xi|^2}{4}\right)^{1/2}\eta_1+ \left(\tau^2+ \omega^2 \varepsilon_0 \mu_0\right)^{1/2}\eta_2 \label{zeta1def_loc}, \\
\zeta_2=\frac{1}{2}\xi-i\left(\tau^2+\frac{|\xi|^2}{4}\right)^{1/2}\eta_1+ \left(\tau^2+ \omega^2 \varepsilon_0 \mu_0\right)^{1/2}\eta_2 \label{zeta2def_loc},
\end{gather}
with $ \tau \geq 1 $ a free parameter controlling the size of $ |\zeta_1| $ and $ |\zeta_2| $, where $ \xi, \eta_1,\eta_2 $ constant vector fields satisfying $ |\eta_1| = |\eta_2| = 1 $, $ \eta_1 \cdot \eta_2 = 0 $, $ \eta_j \cdot \xi = 0 $ for $ j = 1 , 2 $ and $ \xi \neq e_3 $. More precisely, if $ \xi $ reads in the coordinates $ \mathcal{E} $ as
\begin{equation*}
\xi = \left(\begin{array}{c c c}
\xi^{(1)} & \xi^{(2)} & \xi^{(3)}
\end{array}\right)^t,
\end{equation*}
we choose
\begin{equation*}
\eta_1 = \frac{1}{|\xi'|} \left(\begin{array}{c}
\xi^{(2)} \\ -\xi^{(1)} \\ 0
\end{array}\right)
\qquad
\eta_2 = \eta_1 \times \frac{\xi}{|\xi|} = \frac{1}{|\xi'||\xi|} \left(\begin{array}{c}
-\xi^{(1)} \xi^{(3)} \\ -\xi^{(2)} \xi^{(3)} \\ |\xi'|^2
\end{array}\right),
\end{equation*}
with $ |\xi'|^2 = (\xi^{(1)})^2 + (\xi^{(2)})^2 $.
Observe that $ \zeta_1 - \overline{\zeta_2} = - \xi $ and
\begin{equation*}
\frac{\zeta_1}{|\zeta_1|} = i \frac{\eta_1}{\sqrt{2}} + \frac{\eta_2}{\sqrt{2}} + \mathcal{O}(\tau^{-1}), \qquad \frac{\zeta_2}{|\zeta_2|} = -i \frac{\eta_1}{\sqrt{2}} + \frac{\eta_2}{\sqrt{2}} + \mathcal{O}(\tau^{-1}).
\end{equation*}
We now choose other euclidean coordinates $ \mathcal{F} $ by fixing the following orthonormal basis of $ \mathbb{R}^3 $:
\begin{equation*}
f_2 = \frac{1}{|\xi '|} \left( \begin{array}{c}
\xi^{(1)} \\ \xi^{(2)} \\ 0
\end{array} \right),
\quad
f_3 = \left( \begin{array}{c}
0 \\ 0 \\ 1
\end{array} \right) = e_3,
\quad
f_1 = f_2 \times f_3.
\end{equation*}
Here the vectors $ f_1, f_2, f_3 $ are expressed in the coordinates $ \mathcal{E} $. In these new coordinates $ \xi $, $ \eta_1 $ and $ \eta_2 $ read as
\begin{equation*}
\xi = \left( \begin{array}{c}
0 \\ |\xi '| \\ \xi^{(3)}
\end{array} \right),
\quad
\eta_1 = \left( \begin{array}{c}
1 \\ 0 \\ 0
\end{array} \right),
\quad
\eta_2 = \frac{1}{|\xi|} \left( \begin{array}{c}
0 \\ -\xi^{(3)} \\ |\xi '|
\end{array} \right).
\end{equation*}
Obviously, the metric $ e $ in these coordinates is still the identity matrix.

Therefore, $ \zeta_1 $ and $ \zeta_2 $ read in these coordinates $ \mathcal{F} $ as
\begin{equation*}
\zeta_1 = \left( \begin{array}{c}
i\Big( \tau^2 + \frac{|\xi|^2}{4} \Big)^{1/2} \\ -\frac{|\xi'|}{2}-(\tau^2 + k^2)^{1/2} \frac{\xi^{(3)}}{|\xi|} \\ -\frac{\xi^{(3)}}{2} + (\tau^2 + k^2)^{1/2}\frac{|\xi'|}{|\xi|}
\end{array} \right)
\quad
\zeta_2 = \left( \begin{array}{c}
-i\Big(\tau^2+\frac{|\xi|^2}{4}\Big)^{1/2} \\ \frac{|\xi'|}{2}-(\tau^2+k^2)^{1/2}\frac{\xi^{(3)}}{|\xi|} \\ \frac{\xi^{(3)}}{2}+(\tau^2+k^2)^{1/2}\frac{|\xi'|}{|\xi|}
\end{array} \right),
\end{equation*}
where $ k^2 = \omega^2 \varepsilon_0 \mu_0 $.

Consider $ Z_1 = e^{i \zeta_1 \cdot x} (L_1 + R_1), Y_1 $ the solution stated in Proposition \ref{pro:CGO-sch} corresponding to the pair $ \tilde{\mu}_1, \tilde{\gamma}_1 $ with $ |\zeta_1| > C(\rho, M) $. Recall that
\begin{equation*}
L_1 = \frac{1}{|\zeta_1|} \frv{\zeta_1 \cdot A_1}{\omega \varepsilon_0^{1/2} \mu_0^{1/2} B_1}{\zeta_1 \cdot B_1}{\omega \varepsilon_0^{1/2} \mu_0^{1/2} A_1},\qquad
\norm{R_1}{}{L^2 (\inte; \mathcal{Y})} \leq \frac{C(\rho, \inte, M)}{|\zeta_1|}.
\end{equation*}
Additionally, consider $ Y_2 = e^{i \zeta_2 \cdot x} (M_2 + S_2) $ the solutions stated in Proposition \ref{pro:CGO-rescale*} corresponding to $ \tilde{\mu}_2, \tilde{\gamma}_2 $ with $ |\zeta_2| > C(\rho, M) $. Also recall that
\begin{equation*}
M_2 = \frac{1}{|\zeta_2|} \frv{\zeta_2 \cdot A_2}{\cross{-\zeta_2}{A_2}}{\zeta_2 \cdot B_2}{\cross{\zeta_2}{B_2}},\qquad
\norm{S_2}{}{L^2(\inte; \mathcal{Y})} \leq \frac{C(\rho, \inte, M)}{|\zeta_2|}.
\end{equation*}
Before plugging $ Z_1 $ and $ \mathpzc{Y}_2 = Y_2 - \dot{Y}_2 $ into the estimate given in Proposition \ref{le:2-suitable-estimate_loc}, we establish a quantitative version of the Riemann-Lebesgue lemma.

\begin{lemma}\label{le:RLlemma}\rm
Let $ \tau $ be a positive parameter, $ q \in L^1 (\mathbb{R}^n) $ and
\[ \omega_q(r) := \sup_{ |y| < r }\norm{q - q(\centerdot - y)}{}{L^1(\mathbb{R}^n)}. \]
Consider $ \phi(\centerdot; \tau) \in C^1(\mathbb{R}^n; \mathbb{R}) $, then for any $  0 < \mathpzc{d} < 1 $ one has
\[ \left| \int_{\mathbb{R}^n} e^{i\phi(\centerdot; \tau)} q \,dV \right| \leq \omega_q(\mathpzc{d}) + C \mathpzc{d}^{-1} \sup_{x \in \mathbb{R}^n} \frac{1 + |\nabla \phi(x; \tau)|}{1 + |\nabla \phi(x; \tau)|^2} \norm{q}{}{L^1(\mathbb{R}^n)}. \]
\end{lemma}

\begin{proof}
Take $ \varphi \in C^\infty_0 (\mathbb{R}^n; \mathbb{R}_+) $ such that $ \norm{\varphi}{}{L^1(\mathbb{R}^n)} = 1 $ with
\[ \mathrm{supp}\, \varphi \subset \{ x \in \mathbb{R}^n : |x| < 1 \} \]
and denote $ \varphi_\mathpzc{d} = \mathpzc{d}^{-n} \varphi (\centerdot/\mathpzc{d}) $. Then one has that
\[ \int_{\mathbb{R}^n} e^{i\phi(\centerdot; \tau)} q \,dx = \int_{\mathbb{R}^n} e^{i\phi(\centerdot; \tau)} ( q - \varphi_\mathpzc{d} \ast q ) \,dx + \int_{\mathbb{R}^n} e^{i\phi(\centerdot; \tau)} \varphi_\mathpzc{d} \ast q \,dx. \]
On one hand,
\begin{equation*}
\norm{q - \varphi_\mathpzc{d} \ast q}{}{L^1(\mathbb{R}^n)} \leq \int_{\mathbb{R}^n} \varphi(y) \norm{q - q(\centerdot - \mathpzc{d} y)}{}{L^1(\mathbb{R}^n)} \,dy.
\end{equation*}
On the other hand, since
\[ (1 + \nabla \phi(\centerdot; \tau) \cdot D) e^{i \phi(\centerdot; \tau)} = (1 + |\nabla \phi(\centerdot; \tau)|^2) e^{i \phi(\centerdot; \tau)} \]
one has integrating by parts
\[ \int_{\mathbb{R}^n} e^{i\phi(\centerdot; \tau)} \varphi_\mathpzc{d} \ast q \,dx = \int_{\mathbb{R}^n} e^{i\phi(\centerdot; \tau)} \left( \frac{1 - \nabla \phi(\centerdot; \tau) \cdot D}{1 + |\nabla \phi(\centerdot; \tau)|^2} \right) (\varphi_\mathpzc{d} \ast q) \,dx, \]
hence
\[ \left| \int_{\mathbb{R}^n} e^{i\phi(\centerdot; \tau)} \varphi_\mathpzc{d} \ast q \,dx \right| \leq C \mathpzc{d}^{-1} \sup_{x \in \mathbb{R}^n} \frac{1 + |\nabla \phi(x; \tau)|}{1 + |\nabla \phi(x; \tau)|^2} \norm{q}{}{L^1(\mathbb{R}^n)}. \]
\end{proof}

Recall that $ \mu_j, \gamma_j \in H^{2+s}(\inte) $ with $ 0 < s < 1/2 $. In particular, $ \partial^\alpha \mu_j, \partial^\alpha \gamma_j $ are in $ H^s(\inte) $ for $ 0 \leq |\alpha| \leq 2 $. Moreover, the extension by zero allows to identify $ H^t(\inte) = H^t_0(\inte) $ for $ -1/2 < t < 1/2 $ (see \cite{T}). Hence, $ \mathbf{1}_\inte(\partial^\alpha \mu_j) $ and $ \mathbf{1}_\inte(\partial^\alpha \gamma_j) $ are in $ H^s_0(\inte) $ which implies
\begin{align}
\label{es:L1-modulus_of_cont_gamma}
& \sup_{|y| < r} \norm{\mathbf{1}_\inte\partial^\alpha \gamma_j - (\mathbf{1}_\inte\partial^\alpha \gamma_j)(\centerdot - y)}{}{L^1(\mathbb{R}^3)} \leq C r^s, & \\
\label{es:L1-modulus_of_cont_mu}
& \sup_{|y| < r} \norm{\mathbf{1}_\inte\partial^\alpha \mu_j - (\mathbf{1}_\inte\partial^\alpha \mu_j)(\centerdot - y)}{}{L^1(\mathbb{R}^3)} \leq C r^s. &
\end{align}
This property can be found in \cite{St}.

Now we plugging $ Z_1 $ and $ \mathpzc{Y}_2 = Y_2 - \dot{Y}_2 $ into the estimate given in Proposition \ref{le:2-suitable-estimate_loc} with different choices of $ A_j, B_j $.

Choosing $ B_1 = B_2 = 0 $ and $ A_1, A_2 $ such that
\[ \left( i \frac{\eta_1}{\sqrt{2}} + \frac{\eta_2}{\sqrt{2}} \right) \cdot A_1 = \left( i \frac{\eta_1}{\sqrt{2}} + \frac{\eta_2}{\sqrt{2}} \right) \cdot \overline{A_2} = 1 \]
one gets, when $ \tau $ becomes large, that
\begin{gather*}
\Inner{(Q_1 - Q_2) Z_1}{Y_2}_\inte =\\
= \int_\inte e^{- i \xi \cdot x} \left( \frac{1}{2} \Delta (\alpha_1 - \alpha_2) + \frac{1}{4} ( \inner{\nabla \alpha_1}{\nabla \alpha_1} - \inner{\nabla \alpha_2}{\nabla \alpha_2} ) + ( \kappa^2_2 - \kappa^2_1 ) \right) \, dV\\
+\, \mathcal{O}((\tau^2 + |\xi|^2)^{-1/2}),
\end{gather*}
where the implicit constant is $ C(\rho, \inte, M) $. In addition, by Lemma \ref{le:RLlemma} with
\[ \phi(x; \tau) = -|\xi'|x^2 + 2(\tau^2 + k^2)^{1/2} \frac{|\xi'|}{|\xi|}x^3 \]
and (\ref{es:L1-modulus_of_cont_gamma}), (\ref{es:L1-modulus_of_cont_mu}) one has
\begin{equation*}
\Inner{(Q_1 - Q_2) Z_1}{\dot Y_2}_\inte
= \mathcal{O} \left( \mathpzc{d}^s + \frac{\mathpzc{d}^{-1} |\xi|}{(|\xi|^2 + |\xi|^2|\xi'|^2 + 4(\tau^2 + k^2) |\xi'|^2)^{1/2}} \right).
\end{equation*}
Choosing $ A_1 = A_2 = 0 $ and $ B_1, B_2 $ such that
\[ \left( i \frac{\eta_1}{\sqrt{2}} + \frac{\eta_2}{\sqrt{2}} \right) \cdot B_1 = \left( i \frac{\eta_1}{\sqrt{2}} + \frac{\eta_2}{\sqrt{2}} \right) \cdot \overline{B_2} = 1 \]
one gets, when $ \tau $ becomes large, that
\begin{gather*}
\Inner{(Q_1 - Q_2) Z_1}{Y_2} =\\
= \int_\inte e^{- i \xi \cdot x} \left( \frac{1}{2} \Delta (\beta_1 - \beta_2) + \frac{1}{4} ( \inner{\nabla \beta_1}{\nabla \beta_1} - \inner{\nabla \beta_2}{\nabla \beta_2} ) + ( \kappa^2_2 - \kappa^2_1 ) \right) \, dV\\
+\, \mathcal{O}((\tau^2 + |\xi|^2)^{-1/2}),
\end{gather*}
where the implicit constant is $ C(\rho, \inte, M) $. Again, by Lemma \ref{le:RLlemma} and (\ref{es:L1-modulus_of_cont_gamma}), (\ref{es:L1-modulus_of_cont_mu}) one has
\begin{equation*}
\Inner{(Q_1 - Q_2) Z_1}{\dot Y_2}_\inte
= \mathcal{O} \left( \mathpzc{d}^s + \frac{\mathpzc{d}^{-1} |\xi|}{(|\xi|^2 + |\xi|^2|\xi'|^2 + 4(\tau^2 + k^2) |\xi'|^2)^{1/2}} \right).
\end{equation*}
Writing
\begin{align*}
f &= \mathbf{1}_\inte \left( \frac{1}{2} \Delta (\alpha_1 - \alpha_2) + \frac{1}{4} ( \inner{\nabla \alpha_1}{\nabla \alpha_1} - \inner{\nabla \alpha_2}{\nabla \alpha_2} ) + ( \kappa^2_2 - \kappa^2_1 ) \right)\\
g &= \mathbf{1}_\inte \left( \frac{1}{2} \Delta (\beta_1 - \beta_2) + \frac{1}{4} ( \inner{\nabla \beta_1}{\nabla \beta_1} - \inner{\nabla \beta_2}{\nabla \beta_2} ) + ( \kappa^2_2 - \kappa^2_1 ) \right),
\end{align*}
where $ \mathbf{1}_\inte $ is the indicator function of $ \inte $. By Proposition \ref{le:2-suitable-estimate_loc} and the properties of the special solutions, there exist three constants $ c = c(\inte) $, $ C = C(\rho, \inte, M) $ and $ C' = C'(\rho, M) $ such that, for any $ \tau \geq C' $ one has
\begin{gather*}
|\widehat{f}(\xi)| + |\widehat{g}(\xi)| \leq C \left( B\big( \delta_C(C_\Gamma^1, C_\Gamma^2) \big) e^{c(\tau^2 + |\xi|^2)^{1/2}} + (\tau^2 + |\xi|^2)^{-1/2} \right.\\
\left. +\, \mathpzc{d}^s + \frac{\mathpzc{d}^{-1} |\xi|}{(|\xi|^2 + |\xi|^2|\xi'|^2 + 4(\tau^2 + k^2) |\xi'|^2)^{1/2}} \right).
\end{gather*}

Note that, for $ R \geq 1 $, one has
\begin{gather*}
\norm{f}{2}{H^{-1}(\mathbb{E})} + \norm{g}{2}{H^{-1}(\mathbb{E})} = \int_{|\xi| < R} (1 + |\xi|^2)^{-1} \big( |\widehat{f}(\xi)|^2 + |\widehat{g}(\xi)|^2 \big) \, d\xi\\
+ \int_{|\xi| \geq R} (1 + |\xi|^2)^{-1} \big( |\widehat{f}(\xi)|^2 + |\widehat{g}(\xi)|^2 \big) \, d\xi\\
\leq C \left( B\big( \delta_C(C_1, C_2) \big) e^{c(R + \tau)} +\, \tau^{-1} + \mathpzc{d}^s \right)^2 \int_0^R (1 + |r|^2)^{-1} r^2 \, dr\\
+ C \int_{|\xi| < R} (1 + |\xi|^2)^{-1} \frac{\mathpzc{d}^{-2} |\xi|^2}{|\xi|^2 + |\xi|^2|\xi'|^2 + 4(\tau^2 + k^2) |\xi'|^2} \, d\xi\\
+\, (1 + R^2)^{-1} \left( \norm{f}{2}{L^2(\inte)} + \norm{g}{2}{L^2(\inte)} \right).
\end{gather*}

\begin{lemma}\rm
One has that
\[ \int_{|\xi| < R} (1 + |\xi|^2)^{-1} \frac{\mathpzc{d}^{-2} |\xi|^2}{|\xi|^2 + |\xi|^2|\xi'|^2 + 4(\tau^2 + k^2) |\xi'|^2} \, d\xi \leq C \frac{R}{\mathpzc{d}^2 \tau}. \]
\end{lemma}

\begin{proof} Since $ \{ \xi \in \mathbb{R}^3 : |\xi| < R \} \subset \{ \xi \in \mathbb{R}^3 : |\xi'| < R, |\xi^{(3)}| < R \} $, the integral in the statement is bounded by
\[  \int_{|\xi'| < R} \int_{|\xi^{(3)}| < R} (1 + |\xi|^2)^{-1} \frac{\mathpzc{d}^{-2} |\xi|^2}{|\xi|^2 + |\xi|^2|\xi'|^2 + 4(\tau^2 + k^2) |\xi'|^2} \, d\xi^{(3)}\, d\xi'. \]
Changing to cylindrical coordinates it is enough to study
\begin{equation*}
I(R, \tau) := \int_{[0, R] \times [0, R]} (1 + r^2 + t^2)^{-1} \frac{r (r^2 + t^2)}{r^2 + t^2 + (r^2 + t^2)r^2 + 4(\tau^2 + k^2) r^2} \, dt\,dr.
\end{equation*}
One has
\begin{gather*}
I(R, \tau) \leq \int_{[0, R] \times [0, R]} (1 + r^2 + t^2)^{-1} \frac{(r^2 + t^2)^{1/2}}{(r^2 + t^2 + (r^2 + t^2)r^2)^{1/2}}\\
\times \frac{r (r^2 + t^2)^{1/2}}{((r^2 + t^2)r^2 + 4(\tau^2 + k^2) r^2)^{1/2}} \, dt\,dr \\
= \int_{[0, R] \times [0, R]} (1 + r^2 + t^2)^{-1} \frac{1}{(1 + r^2)^{1/2}} \frac{r (r^2 + t^2)^{1/2}}{((r^2 + t^2)r^2 + 4(\tau^2 + k^2) r^2)^{1/2}} \, dt\,dr \\
\leq \sqrt[]{2}  \int_{0 < r < t < R} (1 + r^2 + t^2)^{-1} \frac{rt}{(t^2r^2 + 4(\tau^2 + k^2) r^2)^{1/2}} \, dt\,dr \\
+\, \sqrt[]{2} \int_{0 < t < r < R} (1 + r^2 + t^2)^{-1} \frac{r^2}{(r^2r^2 + 4(\tau^2 + k^2) r^2)^{1/2}} \, dt\,dr 
\end{gather*}
\begin{gather*}
= 2\, \sqrt[]{2}  \int_{0 < r < t < R} (1 + r^2 + t^2)^{-1} \frac{t}{(t^2 + 4(\tau^2 + k^2))^{1/2}} \, dt\,dr \\
= 2\, \sqrt[]{2}  \int_{[0, R]} \frac{t}{(t^2 + 4(\tau^2 + k^2))^{1/2}} \frac{1}{(1 + t^2)^{1/2}} \arctan{\left(\frac{t}{(1 + t^2)^{1/2}}\right)} \, dt \leq C \frac{R}{\tau}.
\end{gather*}
\end{proof}

Therefore,
\begin{gather*}
\norm{f}{}{H^{-1}(\mathbb{E})} + \norm{g}{}{H^{-1}(\mathbb{E})} \leq C \left( B\big( \delta_C(C_\Gamma^1, C_\Gamma^2) \big) e^{c(R + \tau)} + \tau^{-1}  R^{1/2} \right. \\
\left. + \mathpzc{d}^s R^{1/2} + R^{1/2} \mathpzc{d}^{-1} \tau^{-1/2} + R^{-1} \right).
\end{gather*}

Now we choose $ R $ in such a way that $ \mathpzc{d}^s R^{1/2} + R^{1/2} \mathpzc{d}^{-1} \tau^{-1/2} $ behaves as $ R^{-1} $, that is,
\[ R = \frac{ \mathpzc{d}^{2/3} \tau^{1/3} }{ (1 + \mathpzc{d}^{1+s} \tau^{1/2})^{2/3} } ,\]
hence
\begin{equation*}
\norm{f}{}{H^{-1}(\mathbb{E})} + \norm{g}{}{H^{-1}(\mathbb{E})} \leq C \left( B\big( \delta_C(C_\Gamma^1, C_\Gamma^2) \big) e^{c \tau} + \left( \mathpzc{d}^s + \frac{1}{\mathpzc{d} \tau^{1/2}} \right)^{2/3} \right).
\end{equation*}
Choosing $ \tau = \mathpzc{d}^{-2(1+s)} $ the estimate becomes
\begin{equation*}
\norm{f}{}{H^{-1}(\mathbb{E})} + \norm{g}{}{H^{-1}(\mathbb{E})} \leq C \left( B\big( \delta_C(C_\Gamma^1, C_\Gamma^2) \big) e^{c \mathpzc{d}^{-2(1+s)}} + \mathpzc{d}^{\frac{2s}{3}} \right).
\end{equation*}
On the other hand, the a priori bound was chosen to have
\[ \norm{f}{}{H^{s}(\inte)} + \norm{g}{}{H^{s}(\inte)} \leq C(M), \]
for $ 0 < s < 1/2 $. Finally, interpolation theory ensures the existence of two constants $ C' = C'(\rho, M) $ and $ C = C(\rho, \inte, M, \omega) $ such that, for any $ \mathpzc{d} \leq C' $, the following estimate holds
\begin{equation}\label{es:normL2de_f_g:locd}
\norm{f}{}{L^2(\inte)} + \norm{g}{}{L^2(\inte)} \leq C \left( B\big( \delta_C(C_\Gamma^1, C_\Gamma^2) \big) e^{c \mathpzc{d}^{-2(1+s)}} + \mathpzc{d}^{\frac{2s}{3}} \right)^\theta
\end{equation}
with $ 0 = -\theta + (1 - \theta) s $.

The idea now is to transfer this estimate from $ f, g $ to the difference of the coefficients $ \tilde{\mu}_1 - \tilde{\mu}_2 $ and $ \tilde{\gamma}_1 - \tilde{\gamma}_2 $. This can be accomplished by using the following Carleman estimate.

There exists a positive constant $ C(\inte) $ such that, for all $ h \leq 1 $ and any function $ \phi $ smooth enough, the following estimate holds
\begin{gather*}
h \norm{e^{\varphi / h}\phi}{2}{L^2(\inte)} + h^3 \norm{e^{\varphi / h} \nabla \phi}{2}{L^2(\inte; \mathbb{C}^3)} \leq\\
\leq C \left( h^4 \norm{e^{\varphi / h} \Delta \phi}{2}{L^2(\inte)} + h \norm{e^{\varphi / h}\phi}{2}{L^2(\bou)} + h^3 \norm{e^{\varphi / h} \nabla \phi}{2}{L^2(\bou; \mathbb{C}^3)} \right),
\end{gather*}
where $ \varphi = 1/2 |x-x_0|^2 $ with $ x_0 \notin \overline{\inte} $. The constant here depends on the distance from $ x_0 $ to $ \inte $ and on the diameter of $ \inte $. A Carleman estimate of this type can be found in \cite{I}.

A simple computation give:
\begin{align*}
f &= \mathbf{1}_\inte \tilde{\gamma}_1^{-1/2} \left[ \Delta(\tilde{\gamma}_1^{1/2} - \tilde{\gamma}_2^{1/2}) + q_f(\tilde{\gamma}_1^{1/2} - \tilde{\gamma}_2^{1/2}) + p_f(\tilde{\mu}_1^{1/2} - \tilde{\mu}_2^{1/2}) \right],\\
g &= \mathbf{1}_\inte \tilde{\mu}_1^{-1/2} \left[ \Delta(\tilde{\mu}_1^{1/2} - \tilde{\mu}_2^{1/2}) + q_g(\tilde{\mu}_1^{1/2} - \tilde{\mu}_2^{1/2}) + p_g(\tilde{\gamma}_1^{1/2} - \tilde{\gamma}_2^{1/2}) \right];
\end{align*}
where
\begin{align*}
q_f = - \left( \frac{\Delta \tilde{\gamma}_2^{1/2}}{\tilde{\gamma}_2^{1/2}} + \omega^2 \tilde{\gamma}_1^{1/2}(\tilde{\gamma}_1^{1/2} \tilde{\mu}_1 + \tilde{\gamma}_2^{1/2} \tilde{\mu}_2) \right),& &
p_f = - \omega^2 \tilde{\gamma}_1 \tilde{\gamma}_2^{1/2} (\tilde{\mu}_1^{1/2} + \tilde{\mu}_2^{1/2}),\\
q_g = - \left( \frac{\Delta \tilde{\mu}_2^{1/2}}{\tilde{\mu}_2^{1/2}} + \omega^2 \tilde{\mu}_1^{1/2}(\tilde{\mu}_1^{1/2} \tilde{\gamma}_1 + \tilde{\mu}_2^{1/2} \tilde{\gamma}_2) \right),& &
p_g = - \omega^2 \tilde{\mu}_1 \tilde{\mu}_2^{1/2} (\tilde{\gamma}_1^{1/2} + \tilde{\gamma}_2^{1/2}).
\end{align*}
Note that, thanks to the a priori bound, we have the following differential inequalities:
\begin{align*}
|\Delta(\tilde{\gamma}_1^{1/2} - \tilde{\gamma}_2^{1/2})| \leq C(M) (|f| + |\tilde{\gamma}_1^{1/2} - \tilde{\gamma}_2^{1/2}| + |\tilde{\mu}_1^{1/2} - \tilde{\mu}_2^{1/2}|),\\
|\Delta(\tilde{\mu}_1^{1/2} - \tilde{\mu}_2^{1/2})| \leq C(M) (|g| + |\tilde{\gamma}_1^{1/2} - \tilde{\gamma}_2^{1/2}| + |\tilde{\mu}_1^{1/2} - \tilde{\mu}_2^{1/2}|).
\end{align*}

In order to simplify the notation, we shall write $ \phi_1 = \tilde{\gamma}_1^{1/2} - \tilde{\gamma}_2^{1/2} $ and $ \phi_2 = \tilde{\mu}_1^{1/2} - \tilde{\mu}_2^{1/2} $. By the differential inequalities written above and the Carleman estimate, one has
\begin{gather*}
\sum_{j=1,2} \left( h \norm{e^{\varphi / h}\phi_j}{2}{L^2(\inte)} + h^3 \norm{e^{\varphi / h} \nabla \phi_j}{2}{L^2(\inte; \mathbb{C}^3)} \right) \leq\\
\leq C'' \sum_{j=1,2} \left( h^4 \norm{e^{\varphi / h} \phi_j}{2}{L^2(\inte)} + h \norm{e^{\varphi / h} \phi_j }{2}{L^2(\bou)} + h^3 \norm{e^{\varphi / h} \nabla \phi_j}{2}{L^2(\bou; \mathbb{C}^3)} \right)\\
+ C'' h^4 \left( \norm{e^{\varphi / h} f}{2}{L^2(\inte)} + \norm{e^{\varphi / h} g}{2}{L^2(\inte)} \right),
\end{gather*}
where the constant is $ C'' = C''(\inte, M) $ and $ \varphi(x) = 1/2 |x-x_0|^2 $ with $ x_0 \notin \overline{\inte} $. The terms $ h^4 \norm{e^{\varphi / h} \phi_j}{2}{L^2(\inte)} $, with $ j = 1,2 $, can be absorbed by the left hand side of the inequality. Hence, if $ d_1 = \textrm{inf} \{ |x - x_0|^2 : x \in \inte \} $ and $ d_2 = \textrm{sup} \{ |x - x_0|^2 : x \in \inte \} $ we get, for any $ h < C''(\inte, M)^{-1/3} $, that
\begin{gather*}
e^{d_1 / h} \sum_{j=1,2} \left( h \norm{\phi_j}{2}{L^2(\inte)} + h^3 \norm{\nabla \phi_j}{2}{L^2(\inte; \mathbb{C}^3)} \right) \leq C'' e^{d_2 / h}\times\\
\left[ h^4 \left( \norm{f}{2}{L^2(\inte)} + \norm{g}{2}{L^2(\inte)} \right) + \sum_{j=1,2} \left( h \norm{\phi_j }{2}{L^2(\bou)} + h^3 \norm{\nabla \phi_j}{2}{L^2(\bou; \mathbb{C}^3)} \right)\right ].
\end{gather*}
But now we can easily estimate
\begin{gather*}
\norm{\phi_1}{}{L^2(\bou)} + \norm{\nabla \phi_1}{}{L^2(\bou; \mathbb{C}^3)} \leq C B \big( \delta_C(C_\Gamma^1, C_\Gamma^2) \big),\\
\norm{\phi_2}{}{L^2(\bou)} + \norm{\nabla \phi_2}{}{L^2(\bou; \mathbb{C}^3)} \leq C B \big( \delta_C(C_\Gamma^1, C_\Gamma^2) \big),\\
\norm{\tilde{\gamma}_1 - \tilde{\gamma}_2}{}{L^2(\inte)} + \norm{\nabla (\tilde{\gamma}_1 - \tilde{\gamma}_2)}{}{L^2(\inte; \mathbb{C}^3)} \leq C \left( \norm{\phi_1}{}{L^2(\inte)} + \norm{\nabla \phi_1}{}{L^2(\inte; \mathbb{C}^3)} \right),\\
\norm{\tilde{\mu}_1 - \tilde{\mu}_2}{}{L^2(\inte)} + \norm{\nabla (\tilde{\mu}_1 - \tilde{\mu}_2)}{}{L^2(\inte; \mathbb{C}^3)} \leq C \left( \norm{\phi_2}{}{L^2(\inte)} + \norm{\nabla \phi_2}{}{L^2(\inte; \mathbb{C}^3)} \right).
\end{gather*}
The constants above depend on the a priori bounds $ M $. With these inequalities and estimate (\ref{es:normL2de_f_g:locd}), we obtain
\begin{gather*}
\norm{\tilde{\gamma}_1 - \tilde{\gamma}_2}{}{H^1(\inte)} + \norm{\tilde{\mu}_1 - \tilde{\mu}_2}{}{H^1(\inte)} \leq C e^{\frac{d_2 - d_1}{2h}} B\big( \delta_C(C_\Gamma^1, C_\Gamma^2) \big)\\
+ C e^{\frac{d_2 - d_1}{2h}} \left( B\big( \delta_C(C_\Gamma^1, C_\Gamma^2) \big) e^{c \mathpzc{d}^{-2(1+s)}} + \mathpzc{d}^{\frac{2s}{3}} \right)^\frac{s}{1 + s},
\end{gather*}
where $ d_2 > d_1 $, $ 0 < s < 1/2 $, $ C = C(\rho, \inte, M) $, $ \mathpzc{d} \leq C'(\rho, M) $, $ c = c(\inte) $ and $ h < C''(\inte, M)^{-1/3} $. To end up with the estimate given in the statement, it is enough to choose the parameter $ \mathpzc{d} $ as
\[ \mathpzc{d}^{-2(1+s)} = -\frac{1}{2c} \mathrm{log} \, B \big( \delta_C(C_\Gamma^1, C_\Gamma^2) \big), \]
and to note that
\[ 0 < \frac{s^2}{3(s + 1)^2} < \frac{s^2}{3}. \]

\section{The domain $ U $ is partially spherical}
Along this section we assume $ U $ to be a suitable partially spherical domain and we follow the notation in Definition \ref{def:geometryU} and Definition \ref{def:suitableU}. Furthermore, $ n $ and $ \nu $ will denote the outward unit normal forms of $ U $ and $ \inte $, respectively.

The basic idea in this section is to use the Kelvin transform $ \mathcal{K} $ to generalize our result on partially flat domain to the case of partially spherical domain. To achieve this, we study the behavior of Maxwell's equations and the distance $ \delta_C $ under $ \mathcal{K} $. 

Note that $ \mathcal{K} = \mathcal{K}^{-1} $ and $ \mathcal{K} $ is a conformal transformation from $ (\inte, e) $ onto $ (U, e) $:
\[ \mathcal{K}^* e = \frac{r_1^4}{|\centerdot|^4} e, \]
where $ \mathcal{K}^* $ denotes the pull-back of $ \mathcal{K} $.

Let $ \tilde{E} = \mathcal{K}^* E $, $ \tilde{H} = \mathcal{K}^* H $, $ \tilde{\mu} = \mathcal{K}^* \mu $, and $ \tilde{\gamma} = \mathcal{K}^* \gamma $. The following is the transformation law for Maxwell's equations under the Kelvin transform.

\begin{lemma}\label{le:invarianceK}\rm
One has $ E, H \in \HcurlU{} $ is solution of
\[ dH + i \omega \gamma \ast\! E = 0 \qquad dE - i \omega \mu \ast\! H = 0 \]
in $ U $, if and only if,
$ \tilde{E}, \tilde{H} \in \Hcurl{} $ is a solution of
\[ d\tilde{H} + i \omega \tilde{\gamma} \frac{r_1^2}{|\centerdot|^2} \ast\! \tilde{E} = 0 \qquad d\tilde{E} - i \omega \tilde{\mu} \frac{r_1^2}{|\centerdot|^2} \ast\! \tilde{H} = 0 \]
in $ \inte $.
\end{lemma}
\begin{proof}
The proof follows easily from
\begin{equation*}
d \mathcal{K}^* \eta = \mathcal{K}^* d\eta, \quad \mathcal{K}^* (\ast \eta) = \ast_{\mathcal{K}^* e} \mathcal{K}^* \eta, \quad *_{ce} \eta = c^{3/2-k} \ast\! \eta.
\end{equation*}
Here $ \eta $ is $ k $-form and $ c $ is an arbitrary positive smooth function.
\end{proof}

\begin{lemma}\label{le:equivalenceK}\rm
Given $ u_j \in \HcurlU{} $ and $ \tilde{v}_j \in \Hcurl{} $ with $ j= 1,2 $, let us consider $ \tilde{u}_j = \mathcal{K}^* u_j \in \Hcurl{} $ and $ v_j = \mathcal{K}^* \tilde{v}_j \in \HcurlU{} $.
\begin{itemize}
\item[(a)] For any $ z \in B^{1/2}(\bouU; \Lambda^1 T\mathbb{E}) $ one has $ \dual{\ast (n \wedge u_j)}{z} = \dual{\ast (\nu \wedge \tilde{u}_j)}{w}, $ where $ w = \mathcal{K}^* v|_{\bou} \in B^{1/2}(\bou; \Lambda^1 T\mathbb{E}) $ with $ v \in H^1(U; \Lambda^1 T\mathbb{E}) $ such that $ v|_{\bouU} = z $. Furthermore,
\begin{equation}\label{es:equivalenceKform1}
\norm{w}{}{B^{1/2}(\bou; \Lambda^1 T\mathbb{E})} \leq C \norm{z}{}{B^{1/2}(\bouU; \Lambda^1 T\mathbb{E})}.
\end{equation}
For any $ w \in B^{1/2}(\bou; \Lambda^1 T\mathbb{E}) $ one has $ \dual{\ast (\nu \wedge \tilde{v}_j)}{w} = \dual{\ast (n \wedge v_j)}{z}, $ where $ z = \mathcal{K}^* u|_{\bouU} \in B^{1/2}(\bouU; \Lambda^1 T\mathbb{E}) $ with $ u \in H^1(\inte; \Lambda^1 T\mathbb{E}) $ such that $ u|_{\bou} = w $. Moreover,
\begin{equation}\label{es:equivalenceKform2}
\norm{z}{}{B^{1/2}(\bouU; \Lambda^1 T\mathbb{E})} \leq C \norm{w}{}{B^{1/2}(\bou; \Lambda^1 T\mathbb{E})}.
\end{equation}
\item[(b)] For any $ h \in B^{1/2}(\bouU) $ one has $ \dual{\Div \ast\! (n \wedge u_j)}{h} = \dual{\Div \ast\! (\nu \wedge \tilde{u}_j)}{g} $ where $ g = \mathcal{K}^* f|_{\bou} \in B^{1/2}(\bou) $ with $ f \in H^1(U) $ such that $ f|_{\bouU} = h $. Moreover,
\begin{equation}\label{es:equivalenceK1}
\norm{g}{}{B^{1/2}(\bou)} \leq C \norm{h}{}{B^{1/2}(\bouU)}.
\end{equation}
For any $ g \in B^{1/2}(\bou) $ one has $ \dual{\Div \ast\! (\nu \wedge \tilde{v}_j)}{g} = \dual{\Div \ast\! (n \wedge v_j)}{h} $ where $ h = \mathcal{K}^* f|_{\bouU} \in B^{1/2}(\bouU) $ with $ f \in H^1(\inte) $ such that $ f|_{\bou} = g $. Moreover,
\begin{equation}\label{es:equivalenceK2}
\norm{h}{}{B^{1/2}(\bouU)} \leq C \norm{g}{}{B^{1/2}(\bou)}.
\end{equation}
\item[(c)] The following estimates hold
\begin{equation*}
\norm{\ast (n \wedge u_1) - \ast (n \wedge u_2)}{}{TH(\bouU)} \leq C \norm{\ast (\nu \wedge \tilde{u}_1) - \ast (\nu \wedge \tilde{u}_2)}{}{TH(\bou)}
\end{equation*}
and
\begin{equation*}
\norm{\ast (\nu \wedge \tilde{v}_1) - \ast (\nu \wedge \tilde{v}_2)}{}{TH(\bou)} \leq C' \norm{\ast (n \wedge v_1) - \ast (n \wedge v_2)}{}{TH(\bouU)}.
\end{equation*}
\end{itemize}
\end{lemma}

\begin{proof}
The proof of the identities is an immediate consequence of the identities stated in the proof of Lemma \ref{le:invarianceK} and the weak definitions of tangential trace and surface divergence. Proving (\ref{es:equivalenceK1}), (\ref{es:equivalenceK2}) is an easy computation and (\ref{es:equivalenceKform1}), (\ref{es:equivalenceKform2}) follow easily in coordinates from (\ref{es:equivalenceK1}), (\ref{es:equivalenceK2}) and (\ref{es:besov_product}). Finally, the estimates in (c) are a consequence of (a), (b), (\ref{es:equivalent_norm_TH-C1}) and (\ref{es:equivalent_norm_TH-C2}).
\end{proof}

\begin{proposition}\label{pro:estima_distance}\rm
One has that
\[ \delta_C(\tilde{C}^1_{\tilde{\Gamma}}, \tilde{C}^j_{\tilde{\Gamma}}) \leq C \delta_C(C^1_\Gamma, C^2_\Gamma), \]
where
\[ C^j_\Gamma = C(\mu_j, \gamma_j; \Gamma), \quad \tilde{C}^j_{\tilde{\Gamma}} = C \left(\frac{r_1^2}{|\centerdot|^2} \tilde{\mu}_j, \frac{r_1^2}{|\centerdot|^2} \tilde{\gamma}_j; \tilde{\Gamma} \right), \]
with $ j = 1,2 $.
\end{proposition}

\begin{proof}
Considering $ E_j, H_j $ and $ \tilde{E}_j, \tilde{H}_j $ as $ u_j $ and $ \tilde{v}_j $ in the statement of Lemma \ref{le:equivalenceK}, this proposition is a consequence of Lemma \ref{le:invarianceK} and the item (c) in Lemma \ref{le:equivalenceK}.
\end{proof}

In order to end up with the proof in the case that $ U $ is partially spherical, it is enough to use Proposition \ref{pro:estima_distance} and recall that
\begin{align*}
\norm{\mu_1 - \mu_2}{}{H^1(U)} \leq C \norm{\frac{r_1^2}{|\centerdot|^2}(\tilde{\mu}_1 - \tilde{\mu}_2)}{}{H^1(\inte)}\\
\norm{\gamma_1 - \gamma_2}{}{H^1(U)} \leq C \norm{\frac{r_1^2}{|\centerdot|^2}(\tilde{\gamma} - \tilde{\gamma}_2)}{}{H^1(\inte)}.
\end{align*}



\begin{thebibliography}{10}

\bibitem{A}
G.~Alessandrini, Stable determination of the conductivity by boundary measurements, \emph{Appl. Anal.} \textbf{27} (1988), 153--172.

\bibitem{AV}
G.~Alessandrini and S.~Vessella, Lipschitz stability for the inverse conductivity problem, \emph{Adv. Appl. Math.} \textbf{35} (2005) 207--241.

\bibitem{B}
R.~Brown, Global uniqueness in the impedance imaging problem for less regular conductivities, \emph{SIAM J. Math. Anal.} \textbf{27} (1996), 1049--1056.

\bibitem{BuU}
A.~Bukhgeim and G.~Uhlmann, Recovering a potential from partial Cauchy data, \textit{Comm. PDE}, \textbf{27} (2002), 653--668.
 
\bibitem{COSa}
P.~Caro, P.~Ola and M.~Salo, Inverse boundary value problem for Maxwell equations with local data, \emph{Comm. PDE} \textbf{34} (2009), 1425--1464.

\bibitem{C}
P.~Caro, Stable determination of the electromagnetic coefficients by boundary measurements, preprint (2010) \textit{arXiv:1001.4664}.

\bibitem{HeW}
H.~Heck and J.-N.~Wang, Stability estimates for the inverse boundary value problem by partial Cauchy data.  \textit{Inverse Problems}  \textbf{22}  (2006), 1787--1796.

\bibitem{HeW2}
H.~Heck and J.-N.~Wang, Optimal stability estimate of the inverse boundary value problem by partial measurements, preprint (2007) \textit{arXiv:0708.3289v1}.

\bibitem{I}
V.~Isakov, Carleman estimates and applications to inverse problems, \emph{Milan J. Math.} \textbf{72} (2004), 249--271.

\bibitem{I2}
V.~Isakov, On uniqueness in the inverse conductivity problem with local data. \textit{Inverse Probl. Imaging} \textbf{1} (2007), 95--105.

\bibitem{JK}
D.~Jerison and C.~Kenig, The inhomogeneous Dirichlet problem in Lipschitz domains, \emph{J. Funct. Anal.} \textbf{130} (1995), 161--219.

\bibitem{JMcD}
M.~Joshi, S.~R.~McDowall, Total determination of material parameters from electromagnetic boundary information. \textit{Pacific J. Math.} \textbf{193} (2000), 107--129.

\bibitem{KSaU}
C.~E.~Kenig, M.~Salo, and G.~Uhlmann, Inverse problems for the anisotropic Maxwell equations, preprint (2009), \textit{arXiv:0905.3275}.

\bibitem{KSjU} Kenig, C.~E., Sj\"ostrand J., Uhlmann, G., The Calder\'on problem with partial data. \textit{Ann. of Math.} \textbf{165} (2007), 567--591.

\bibitem{Le}
R.~Leis, \emph{Initial boundary value problems in mathematical physics}, Wiley, New York (1986).

\bibitem{McD2}
S.~R.~McDowall, An electromagnetic inverse problem in chiral media. \textit{Trans. Amer. Math. Soc.} \textbf{352} (2000), no. 7, 2993--3013.

\bibitem{M}
M.~Mitrea, Sharp Hodge decomposition, Maxwell's equations, and vector Poisson problems on non-smooth, three-dimensional riemannian manifolds, \emph{Duke Math. J.} \textbf{125} (2004), 467--547.

\bibitem{OPS}
P.~Ola, L.~P\"aiv\"arinta, and E.~Somersalo, An inverse boundary value problem in electrodynamics, \emph{Duke Math. J.} \textbf{70} (1993), 617--653.

\bibitem{OPS2}
P.~Ola, L.~P\"aiv\"arinta, E.~Somersalo, Inverse problems for time harmonic electrodynamics. Inside out: inverse problems and applications, 169--191, \textit{Math. Sci. Res. Inst. Publ.}, \textbf{47}, Cambridge Univ. Press, Cambridge, 2003.

\bibitem{OS}
P.~Ola and E.~Somersalo, Electromagnetic inverse problems and generalized Sommerfeld potentials, \emph{SIAM J. Appl. Math.} \textbf{56} (1996), 1129--1145.

\bibitem{SaTz}
M.~Salo, L.~Tzou, Carleman estimates and inverse problems for Dirac operators. \textit{Math. Ann.} \textbf{344} (2009), 161--184.

\bibitem{SaTz2}
M.~Salo and L.~Tzou, Inverse problems with partial data for a Dirac system: a Carleman estimate approach.
\textit{Adv. Math.} (to appear).

\bibitem{SIsCh}
E.~Somersalo, D.~Isaacson and M.~Cheney, A linearized inverse boundary value problem for Maxwell's equations, \emph{J. Comp. Appl. Math.} \textbf{42} (1992), 123--136.

\bibitem{St}
E.~Stein, \emph{Singular Integrals and Differentiability Properties of Functions}, Princeton University Press (1970).

\bibitem{SyU}
J.~Sylvester,  and G.~Uhlmann, A global uniqueness theorem for an inverse boundary value problem, \emph{Ann. of Math.} \textbf{125} (1987), 153--169.

\bibitem{T}
H.~Triebel, Function spaces in Lipschitz domains and on Lipschitz manifolds. Characteristic functions as pointwise multipliers, \emph{Rev. Mat. Complut.} \textbf{15} (2002), 475--524.

\end{thebibliography}
\end{document}